\documentclass[reqno,11pt,draft]{amsart}
\usepackage{amsmath,amssymb,amsfonts}

\usepackage{diagrams}
\usepackage{verbatim}

\diagramstyle[centredisplay,dpi=600,nohug]
\newif\ifhavegoodboldmath\havegoodboldmathtrue
\newarrow{Tto}----{->}

\title[Quasi-coherent sheaves and stable homotopy]{The derived category of quasi-coherent sheaves and axiomatic stable homotopy}

\author[L. Alonso]{Leovigildo Alonso Tarr\'{\i}o}
\address[L. A. T.]{Departamento de \'Alxebra\\
Facultade de Matem\'a\-ticas\\
Universidade de Santiago de Compostela\\
E-15782  Santiago de Compostela, Spain}
\email{leoalonso@usc.es}

\author[A. Jerem\'{\i}as]{Ana Jerem\'{\i}as L\'opez}
\address[A. J. L.]{Departamento de \'Alxebra\\
Facultade de Matem\'a\-ticas\\
Universidade de Santiago de Compostela\\
E-15782  Santiago de Compostela, Spain}
\email{jeremias@usc.es}

\author[M. P\'erez]{Marta P\'erez Rodr\'{\i}guez}
\address[M. P. R.]{Departamento de Matem\'a\-ticas\\
Esc. Sup. de Enx. Inform\'atica,
Campus de Ourense\\
Universidade de Vigo\\
E-32004 Ourense, Spain}
\email{martapr@uvigo.es}

\author[M. J. Vale]{Mar\'{\i}a J. Vale Gonsalves}
\address[M. J. V.]{Departamento de \'Alxebra\\
Facultade de Matem\'a\-ticas\\
Universidade de Santiago de Compostela\\
E-15782  Santiago de Compostela, Spain}
\email{alvale@usc.es}

\thanks{This work has been partially supported by 
Spain's MEC and E.U.'s
FEDER research project MTM2005-05754 together with Xunta de Galicia's grant PGIDIT06PXIC207056PN}


\date{\today}

\theoremstyle{plain}
\newtheorem{thm}{Theorem}[section]
\newtheorem{lem}[thm]{Lemma}
\newtheorem{cor}[thm]{Corollary}
\newtheorem{prop}[thm]{Proposition}

\theoremstyle{remark}
\newtheorem*{rem}{Remark}

\theoremstyle{definition}
\newtheorem*{defn}{Definition}
\newtheorem*{ex}{Example}

\newtheorem{cosa}[thm]{}

\numberwithin{equation}{thm}

\newcommand{\CA}{\mathcal{A}}

\newcommand{\CE}{{\mathcal E}}
\newcommand{\CF}{{\mathcal F}}
\newcommand{\CG}{{\mathcal G}}
\newcommand{\CH}{{\mathcal H}}
\newcommand{\CI}{{\mathcal I}}
\newcommand{\CJ}{{\mathcal J}}
\newcommand{\CK}{{\mathcal K}}

\newcommand{\CM}{\mathcal{M}}

\newcommand{\CO}{\mathcal{O}}
\newcommand{\CP}{\mathcal{P}}

\newcommand{\CS}{\mathcal{S}}

\newcommand{\FU}{\mathfrak U}
\newcommand{\FV}{\mathfrak V}
\newcommand{\FW}{\mathfrak W}
\newcommand{\FX}{\mathfrak X}
\newcommand{\FY}{\mathfrak Y}
\newcommand{\FZ}{\mathfrak Z}

\newcommand{\SC}{\mathsf{C}}

\newcommand{\CCC}{\boldsymbol{\mathsf{C}}}
\newcommand{\D}{\boldsymbol{\mathsf{D}}}
\newcommand{\K}{\boldsymbol{\mathsf{K}}}
\newcommand{\LL}{\boldsymbol{\mathsf{L}}}
\newcommand{\R}{\boldsymbol{\mathsf{R}}}
\newcommand{\T}{\boldsymbol{\mathsf{T}}}
\newcommand{\A}{\mathsf{A}}

\newcommand{\cc}{\mathsf{c}}

\newcommand{\ts}{\mathsf{t}}
\newcommand{\md}{\text{-}\mathsf{Mod}}

\newcommand{\qc}{\mathsf{qc}}
\newcommand{\qct}{\mathsf{qct}}
\newcommand{\op}{\mathsf{o}}
\newcommand{\ab}{\mathsf{Ab}}

\newcommand{\PR}{\mathbf{P}}
\newcommand{\AF}{\mathbf{A}}

\newcommand{\NN}{\mathbb{N}}
\newcommand{\ZZ}{\mathbb{Z}}

\newcommand{\ia}{{\mathfrak a}}

\newcommand{\dirlim}[1]{\begin{array}[t]{c} {\rm lim}\\[-7.5 pt]
 {\longrightarrow} \\[-7.5 pt] {\scriptstyle {#1}} \end{array}}

\newcommand{\lto}{\longrightarrow}
\newcommand{\xto}{\xrightarrow}
\newcommand{\ot}{\leftarrow}
\newcommand{\lot}{\longleftarrow}

\newcommand{\noqed}{\renewcommand{\qed}{}}

\newcommand{\inc}{\hookrightarrow}
\newcommand{\iso}{\tilde{\to}}

\newcommand{\liso}{\tilde{\lto}}
\newcommand{\losi}{\tilde{\lot}}
\newcommand{\imp}{\Rightarrow}

\newcommand{\st}{\mathsf{Ho}\boldsymbol{\mathsf{Sp}}}

\DeclareMathOperator{\Hom}{Hom}
\DeclareMathOperator{\shom}{\CH\mathit{om}}
\DeclareMathOperator{\rshom}{\R\!\shom}
\DeclareMathOperator{\dhom}{\boldsymbol{\CH}\mathsf{om}}

\DeclareMathOperator{\rhom}{\R{}Hom}

\DeclareMathOperator{\spec}{Spec}
\DeclareMathOperator{\spf}{Spf}

\DeclareMathOperator{\h}{H}
\DeclareMathOperator{\id}{id}

\newcommand{\BL}{{\boldsymbol\Lambda}}
\newcommand{\BG}{{\boldsymbol\Gamma}}

\newcommand{\ie}{{\it i.e.\/} }
\newcommand{\cfr}{{\it cfr.\/} }
\newcommand{\lc}{{\it loc.cit.\/} }

\begin{document}

\begin{abstract} We prove in this paper that for a quasi-compact and semi-separated (non necessarily noetherian) scheme $X$, the derived category of quasi-coherent sheaves over $X$, $\D(\A_\qc(X))$, is a stable homotopy category in the sense of Hovey, Palmieri and Strickland, answering a question posed by Strickland. Moreover we show that it is unital and algebraic. We also prove that for a noetherian semi-separated formal scheme $\FX$, its derived category of sheaves of modules with quasi-coherent torsion homologies $\D_\qct(\FX)$ is a stable homotopy category. It is algebraic but if the formal scheme is not a usual scheme, it is not unital, therefore its abstract nature differs essentially from that of the derived category $\D_\qc(X)$ (which is equivalent to $\D(\A_\qc(X))$) in the case of a usual scheme.
\end{abstract}

\maketitle

\section*{Introduction}
A basic structure that arises in homological algebra and homotopy theory is that of \emph{triangulated category}. Part of its axioms were stated by Puppe \cite{P}, and the crucial octahedral axiom was established by Verdier in his thesis (\cite{vtt}, published only recently). It is interesting to note that only the abridged version \cite{vtc} and an account of the theory in \cite{RD} was available in the sixties and the seventies. Though this structure alone gives a lot of tools and permits the generalization of previous results only made previously explicit in algebraic or topological terms (the book \cite{Ntc} is a nice example), it is very frequent that in real applications triangulated categories come together with a richer structure. Two essential examples that come to mind are $\D(R)$, the derived category of complexes of modules over a commutative ring $R$, and $\st$, the category of (non-connective) spectra up to homotopy. There are parallel constructions in both categories that could be, in principle, transported to other contexts. In this vein, Hovey, Palmieri and Strickland  have defined in \cite{hps} the concept of stable homotopy category. It consists of a list of additional properties and structure for a triangulated category.

The rich theory exposed in \lc\!\!\!, together with the list of examples of categories where these axioms are fulfilled show the great interest of this notion. Let us cite some additional examples: categories of equivariant spectra, categories of local spectra (for generalized homology theories like $K(n)$ or $E(n)$), the homotopy category of modules over a ring spectrum, the stable category of representations of a finite group over a field. This examples are taken from the introduction of \cite{str}. In his paper, Strickland points out a conspicuous case that was left out of this list, namely the ``derived category of quasi-coherent sheaves over a nonaffine scheme'' which we will denote $\D(\A_\qc(X))$ for a scheme $X$. Much to our surprise this all-important case had not been settled already in the literature ---not even implicitly.

One issue that arises frequently is that in general, if  $\A$ is an abelian category, its derived category $\D(\A)$ may not have ``small hom-sets'' due to its construction via calculus of fractions.\footnote{Though we prefer the framework of sets and classes \textit{\`a la} von Neumann-G\"odel-Bernays, for those readers fond of universes, it means that to construct $\D(\A)$ we may need to change our universe.} In \cite[Corollary 5.6]{AJS} it is shown that if $\A$ is a Grothendieck category, \ie abelian with a generator and exact filtered directed limits, then $\D(\A)$ exists, \ie it has ``small hom-sets''. This is shown by presenting $\D(\A)$ as a full subcategory of $\K(\A)$ (the category of complexes in $\A$ with maps up to homotopy), a consequence of the fact that $\D(\A)$ is a Bousfield localization of $\D(R)$ where $R$ is the ring of endomorphisms of a generator which is known to have ``small hom-sets'' by an explicit construction of unbounded resolutions. This, together with the fact that $\A_\qc(X)$ is a Grothendieck category, paves the way to check most of the conditions of stable homotopy category. There are, however, at least two conditions whose status does not follow from this, namely, the closed structure and the existence of strongly dualizable objects.

In this paper we prove that the category $\D(\A_\qc(X))$ is a ``stable homotopy category'' in the sense of \cite{hps}, at least when $X$ is quasi-compact and separated. In fact we do not need the full separation hypothesis but a slightly weaker condition called ``semi-separation'', already used by Thomason and his collaborators and students in K-theory of schemes, see \cite{tt}. It was also considered independently with the characterization used here in the context of moduli problems. 

We prove an analogous result for formal schemes, namely, if $\FX$ is a noetherian semi-separated formal scheme the derived category of sheaves with  quasi-coherent torsion homologies, $\D_\qct(\FX)$ (\cfr \cite{dfs}), is a stable homotopy category.

We also show that for a quasi-compact and separated scheme $X$ the category $\D(\A_\qc(X))$ is unital and algebraic in the sense of \cite{hps}. On the other hand if $\FX$ is a noetherian semi-separated formal scheme the category $\D_\qct(\FX)$ is algebraic but most often non unital. This gives a clue on the difference between the structure of the derived categories associated to usual schemes compared to those associated to formal schemes.

There is an aspect that we have not treated at all in this paper, and that seems to be very important for some people working in the field, namely the existence of a good model category whose associated homotopy category is $\D(\A_\qc(X))$. In any case, for us it is enough to use implicitly the DG structure of the category $\CCC(\A_\qc(X))$ that allow us to perform all of the relevant constructions. For the work on model category structures on $\CCC(\A_\qc(X))$, we will content ourselves with mention the papers \cite{hov} and \cite{gill}.

Let us describe in finer detail the contents of this paper. In the first section we recall the axioms in \cite{hps} and the fact that for any scheme $X$ the category $\D(\A_\qc(X))$ is triangulated, possesses coproducts (because we are dealing with unbounded complexes) and that Brown representability holds, being $\A_\qc(X)$ a Grothendieck category.

The second section, which may have an independent interest, gives the essential \emph{sorites} of the semi-separated maps of schemes. They are maps in which the diagonal is affine, though the original definition by Thomason was defined by the existence special bases of open subsets. This point is settled in the remark after Proposition \ref{covss}. A scheme $X$ is semi-separated, \ie the canonical map $X \to \spec(\ZZ)$ is semi-separated, if and only if the intersection of two affine open subsets of $X$ remains affine. This is the consequence of separation that is used in (\v{C}ech) computations in the cohomology of schemes. We have included this section here due to a lack of a complete treatment in the literature.

The third section deals with the closed structure. While it is well-known that the derived tensor product respects quasi-coherence, it is clear that the usual functor $\R{}\shom_X(-,-)$ does not. But it can be fixed composing with the quasi-coherator functor (recalled in \ref{qctor}). This definition of the internal hom functor may seem clumsy, but it is convenient for checking the adjoint property. In the fifth section we show that the internal hom previously defined can be identified with the derived functor of the internal hom in the abelian category $\A_\qc(X)$, which is $Q_X\shom_X(-,-)$ as long as we consider the derived funtor with quasi-coherent K-injective resolutions. This can be considered a technical point and the uninterested reader may skip it.

The fourth section shows the existence of strongly dualizable generators. We proceed recalling the fact proved by Neeman that $\D(\A_\qc(X))$ is compactly generated. In this category a compact object is a perfect complex. Then, we show that a perfect complex is strongly dualizable concluding that $\D(\A_\qc(X))$ is generated by strongly dualizable objects. Together with the previous remarks we obtain our main result stated as Corollary \ref{main}.

The last section deals with the case of formal schemes. Here quasi-coherent torsion sheaves ---introduced and studied in \cite{dfs}--- come into being a natural alternative extending the concept of quasi-coherent sheaves on a ordinary scheme. We prove that for a noetherian semi-separated formal scheme $\FX$ its derived category of sheaves of modules with quasi-coherent torsion homologies, $\D_\qct(\FX)$, is a stable homotopy category following a similar path to the case of usual schemes with an extra twist here and there due to the peculiarities of quasi-coherent torsion sheaves. Note that the lack of \emph{torsion} flat resolutions forces us to work with $\D_\qct(\FX)$ instead of its equivalent subcategory $\D(\A_\qct(\FX))$. Again the closed structure and the existence of strongly dualizable generators are the parts that need to be dealt carefully with.

%

\section{The axioms}

In this section we set the stage by recalling the definition of stable homotopy category and some basic definitions that will be used throughout the paper. We also show that part of the conditions that define a stable homotopy category are well-known to hold for the derived category of quasi-coherent sheaves on any scheme. 

\begin{cosa}\label{axioms} Let us enumerate briefly the five axioms of an abstract stable homotopy category. Let $\T$ be a category, we say that $\T$ is a stable homotopy category if the following hold:
\begin{enumerate}
\item \label{tri} $\T$ is a triangulated category \cite{vtc}.
\item \label{scl} $\T$ is a symmetric closed category in the sense of \cite{lajolla}.
\item \label{sdg} $\T$ possesses a system of \emph{strongly dualizable} generators\footnote{\emph{Generators} in the sense that $\T$ is the smallest localizing subcategory of itself that contains them, see \ref{strong} below.} \cite[Definition 1.1.2]{hps}.
\item \label{cop} $\T$ possesses arbitrary coproducts.
\item \label{brw} A cohomological functor taking values in $\T$ (see below) is representable, \ie it is of the form $\Hom_{\T}(-,X)$ with $X \in \T$.
\end{enumerate}

The definition of \emph{triangulated category} is recalled in \cite[Definition A.1.1.1]{hps}

A \emph{symmetric monoidal category} is a category together with an internal bifunctor (``tensor'' or ``smash'' bifunctor) associative and commutative together with a unit object. This data given up to natural equivalence satisfying certain coherence commutative diagrams. It is said to be \emph{closed} if it possesses another bifunctor (``internal hom'' or ``function space'') suitably adjoint to the previous one. For more details see \cite[Definition A.2.1]{hps} or the original source \cite{lajolla}.

The different notions of generator that will be used are discussed throughout the paper, specially in \ref{genadd} and \ref{strong}.

Denote by $\ab$ the category of abelian groups. A functor from a triangulated category to $\ab$ that takes triangles to exact sequences is often called a homological functor. In \cite{hps} it is called an exact functor. A cohomological functor is a contravariant such functor that takes coproducts to products. To summarize, given a triangulated category $\T$, we say that a functor $\T^{\op} \to \ab$ is cohomological if it takes triangles to exact sequences and if it takes coproducts in $\T$ to products in $\ab$, otherwise said, if the corresponding functor $\T \to \ab^{\op}$ preserves coproducts.
\end{cosa}

Along this paper, for a scheme $(X, \CO_X)$, we will consider the category of sheaves of $\CO_X$-Modules and denote it by $\A(X)$. We will mainly consider its full subcategory of quasi-coherent $\CO_X$-Modules and denote it by $\A_\qc(X)$. For their basic properties see \cite[\S 2.2]{ega1}\footnote{We will add a reference to \cite{ega160} for the benefit of readers using the free \texttt{www.numdam.org} version. Unfortunately, in this case there is no equivalence though the results follow easily from \cite[\S 1]{ega160} together with the well-known fact that a sheaf $\CF$ is quasi-coherent if and only if its restriction to an affine open set $U = \spec(A)$ is such that $\CF|_U \cong \widetilde{M}$ for a certain $A$-module $M$.}.

\begin{cosa} \label{genadd}
On an additive category $\SC$, a set of objects $\{E_{\alpha} \, / \, \alpha \in L\}$ is called a set of generators if, for an object $M \in \SC$, $\Hom_{\SC}(E_{\alpha}, M) = 0$, for all $\alpha \in L$ implies $M = 0$. If $\SC$ possesses coproducts, the existence of a single generator is equivalent to the existence of a set of generators taking as single generator the coproduct of all the objects in the family. Usually, it is more convenient to work with the set of generators than with its coproduct. That is because often, each object in the set may have a finiteness property that is not shared by the whole coproduct.

Recall that an abelian category $\A$ is a Grothendieck category if it possesses a generator and satisfies Grothendieck's axiom AB5 from \cite[1.5]{toh} which is equivalent to the existence of coproducts and the  exactness of filtered direct limits \cfr \cite[Proposition 1.8]{toh}.
\end{cosa}

\begin{lem}\label{grcat}
Let $X$ be a scheme, the category $\A_\qc(X)$ is a Grothendieck category.
\end{lem}

\begin{proof}
First of all, $\A_\qc(X)$ is an abelian category, by \cite[Corollaire (2.2.2)]{ega1}\footnote{See the previous footnote.}. The category $\A(X)$ possesses exact direct limits thus also does $\A_\qc(X)$ because direct limits of quasi-coherent sheaves are quasi-coherent sheaves by \emph{loc.~cit.} Furthermore, direct limits are exact in the full category of sheaves of modules \cite[Proposition 3.1.1]{toh}, in other words $\A_\qc(X)$ is an AB5 category. Finally, it possesses a generator by \cite[Corollary 3.5]{ee}.
\end{proof}

\begin{cosa}
An interesting feature in Enochs and Estrada proof of the existence of a generator for $\A_\qc(X)$ is its expression as the category of quasi-coherent modules over a quiver ring representation. Given a scheme $X$, take a covering by affine open subsets $\{U_\alpha\}_{\alpha \in L_0}$ and cover each $U_\alpha \cap U_\beta$ by affine subsets $\{V_{\alpha \beta \gamma}\}_{\gamma \in L_{\alpha, \beta}}$. Now take a vertex for every $U_\alpha$ and every $V_{\alpha \beta \gamma}$ with $\alpha, \beta \in L_0$ and $\gamma \in L_{\alpha, \beta}$. Draw an edge $\bullet \ot \bullet$ for every inclusion $V_{\alpha \beta \gamma} \inc U_{\alpha}$. This gives a quiver that expresses the underlying structure of the covering. This quiver carries a natural ring representation, take for every vertex corresponding to $U_\alpha$ the ring $\CO_X(U_\alpha)$ and for every vertex corresponding to $V_{\alpha \beta \gamma}$ the ring $\CO_X(V_{\alpha \beta \gamma})$, to every edge take the restriction homomorphism $\CO_X(U_{\alpha}) \to \CO_X(V_{\alpha \beta \gamma})$. Let $R = R(X, \mathbf{U})$ be the ring representation associated to the scheme $X$ and the covering $\mathbf{U} = \{U_{\alpha}, V_{\alpha \beta \gamma}\}_{\alpha, \beta \in L_0, \gamma \in L_{\alpha, \beta}}$. A module over a ring representation of a graph $R$ is the data of an $R(v)$-module $M(v)$ for each vertex $v$ and a morphism $M(r) \colon M(v) \to M(w)$ for each edge $r \colon v \to w$ that is $R(v)$-linear. Such a representation $M$ is called quasi-coherent if for each edge $r$ as above the morphism $1\otimes M(r) \colon R(w) \otimes_{R(v)} M(v) \to M(w)$ is an isomorphism of $R(w)$-modules. Then there is an equivalence of categories between quasi-coherent sheaves on $X$ and quasi-coherent module representations on $R(X, \mathbf{U})$. Actually, what is proved in \cite{ee} is that the category of quasi-coherent module representations of an arbitrary \emph{representation by rings} $R$ of a quiver possesses a generator.

Note that if $X$ is quasi-compact and quasi-separated we may choose a finite quiver to represent $\A_\qc(X)$ as its quasi-coherent module representations. As an example, take $X = \PR^1_K$. Using the obvious two affine open subsets, we obtain the quiver $\bullet \to \bullet \ot \bullet$ 
and its representation 
\[
K[t] \to K[t, t^{-1}] \ot K[t^{-1}].
\]
\end{cosa}

\begin{rem}
The existence of a generator in $\A_\qc(X)$ has been well-known under some hypothesis over the scheme $X$ for a long time. The first to prove it in general was Gabber in an unpublished letter to Conrad \cite[p. 28]{bc}. For a quasi-compact and quasi-separated scheme it is shown to exist in \cite[B.3.]{tt}, that uses \cite[Corollaires (6.9.9), (6.9.12)]{ega1}\footnote{See \cite[(9.4.9)]{ega160} together with \cite[(1.7.7)]{ega41}.}.
\end{rem}

\begin{cosa}
If $\A$ is an abelian category we will denote by $\CCC(\A)$ its category of complexes and by $\D(\A)$ its (unbounded) derived category. Conditions on its homologies will be denoted by a subscript in $\D$. 
We will denote by $\K(\A)$ the homotopy category of complexes of $\A$. As we recalled in the introduction, the construction of $\D(\A)$ via calculus of fractions does not guarantee that it has small hom-sets. However, if $\A$ is a Grothendieck category then $\D(\A)$ has small hom-sets by \cite[Corollary 5.3]{AJS}. This is achieved by identifying $\D(\A)$ with the subcategory of $\K(\A)$ formed by K-injective complexes. A complex $I^\bullet$ is K-injective if for every acyclic complex $A^\bullet \in \K(\A)$ the complex $\Hom^\bullet(A^\bullet, I^\bullet)$ is acyclic. The dual notion is called K-projective. These notions are due to Spaltenstein \cite[1.5]{S}. An important property of K-injective objects is that
\[
\Hom_{\K(\A)}(M^\bullet, I^\bullet) \iso \Hom_{\D(\A)}(M^\bullet, I^\bullet).
\]
for every $M^\bullet \in \K(\A)$.
The notion of K-injective complex is extremely useful as it allows to compute right derived functors. A list of equivalent characterization of these complexes can be found in \cite[Proposition (2.3.8)]{LDC}; note that in \lc K-injective is called $q$-injective.
\end{cosa}

\begin{thm} \label{asht145}
Let $X$ be a scheme. For the category $\D(\A_\qc(X))$ the axioms (\ref{tri}),  (\ref{cop}) and (\ref{brw}) from \ref{axioms} always hold.
\end{thm}

\begin{proof} Most of this assertions are classic. 

The statement (\emph{\ref{tri}}) is trivial because a derived category is triangulated by construction (\cfr \cite[Example (1.4.4)]{LDC}).

The existence of coproducts (\emph{\ref{cop}}) is due to the fact that a coproduct in $\A(X)$ of a family of objects in $\A_\qc(X)$, remains quasi-coherent by \cite[Corollaire (2.2.2)]{ega1}\footnote{\textit{Cfr}. footnote on page \pageref{grcat}.}. But coproducts of sheaves are exact because the category of sheaves of modules $\A(X)$ satisfies AB5 by \cite[Proposition 3.1.1]{toh}, as recalled before. Thus coproducts of complexes represent coproducts in the derived category.

Finally, (\emph{\ref{brw}}) is satisfied because a cohomological functor from the derived category of a Grothendieck category to $\ab$ is representable by \cite[Theorem 5.8]{AJS}.
\end{proof}


\section{Semi-separated maps}

In this section we will discuss a mild generalization of separated maps that it is stronger than quasi-separated and that encompasses the most useful property of separated maps for cohomology. Semi-separated schemes have already been considered in the context of cohomology and $K$-theory by Thomason and Trobaugh \cite{tt} and Thomason's students through a condition on the existence of certain affine bases of the topology of $X$ (see Proposition \ref{sskey} below and the remark after \ref{sssch}). In the context of moduli problems they arise as schemes with affine diagonal, see, for instance, the introduction of \cite{to}. Neither of these references give a systematic development of its main properties nor they give the equivalence between both characterizations, so we do it here.

\begin{defn}
A map of schemes $f \colon X \to Y$ is called \emph{semi-separated} if the associated diagonal map $\Delta_f \colon X \to X \times_Y X$ is affine. For the notion of affine map and its basic properties, see \cite[\S 9.1]{ega1}\footnote{\cite[\S 1]{ega2}}. We will use the notation $\Delta_{X|Y}$ for $\Delta_{f}$ indifferently.
\end{defn}

\begin{rem} Note that a separated morphism is semi-separated because a closed embedding is an affine map \cite[Proposition (9.1.16) (i)]{ega1}. Also, an affine morphism is quasi-compact \cite[Proposition (9.1.3)]{ega1}, so it follows that a semi-separated morphism is quasi-separated.\footnote{You may susbstitute the citations in this paragraph by \cite[Proposition (1.6.2)]{ega2} and \cite[Proposition (1.2.4)]{ega2}}
\end{rem}

\begin{lem}\label{deltafin} 
Let $f \colon X \to Y$ and $g \colon Y \to Z$ be morphisms of schemes. If g is semi-separated, then the canonical morphism $h \colon X \times_Y X \to X \times_Z X$ is affine.
\end{lem}

\begin{proof}
By \cite[Proposition (\textbf{0}.1.4.8)]{ega1} (or use Yoneda's lemma) we have a cartesian square
\begin{diagram}[height=2.2em,w=3em,p=0.3em,labelstyle=\scriptstyle]
X \times_Y X & \rTo^{h}         & X \times_Z X \\
\dTo_p       &                  & \dTo_{f \times_Z f} \\ 
Y            & \rTo^{\Delta_g}  & Y \times_Z Y  \\  
\end{diagram}
where $p$ and $h$ denote the obvious natural morphisms. Now  $\Delta_{g}$ is affine by assumption and therefore $h$ is affine by base-change \cite[Proposition (9.1.16)]{ega1}\footnote{\cite[Proposition (1.6.2)]{ega2}}.
\end{proof}

\begin{prop}\label{sorissep}
In the category of schemes, we have the following:
\begin{enumerate}
\item An embedding (or more generally a radical morphism) is semi-separated.
\item The composition of semi-separated maps is semi-separated.
\item If $f \colon X \to Y$ is a semi-separated $S$-morphism and $g \colon S' \to S$ is a scheme map, then $f_{S'} \colon X \times_S S' \to Y \times_S S'$ is semi-separated.
\item If $f \colon X \to X'$ and $g \colon Y \to Y'$ are semi-separated $S$-morphisms, then $f\times_S g \colon X \times_S Y \to X' \times_S Y'$ is semi-separated.
\item If the composition of two morphisms $g \circ f$ is semi-separated 
then $f$ is semi-separated.
\item If a morphism $f$ is semi-separated then the same is true for its reduced associated morphism $f_\mathrm{red}$.
\end{enumerate}
\end{prop}

\begin{proof}
To see (i), note that a radical morphism, in particular, an embedding is separated by \cite[Proposition (5.3.1), (i)]{ega1}, so it is semi-separated. For (ii), let $f \colon X \to Y$ and $g \colon Y \to Z$ be semi-separated maps. We have the following commutative diagram
\begin{equation}\label{dosdelta}
\begin{diagram}[height=2em,w=2em,p=0.3em,labelstyle=\scriptstyle]
X &                      & \rTo^{\Delta_{X|Z}} &           & X \times_Z X \\
  & \rdTo_{\Delta_{X|Y}} &                     & \ruTo_{h} &\\ 
  &                      & X \times_Y X        &           &\\  
\end{diagram}
\end{equation}
By hypothesis $\Delta_{X|Y}$ is affine and, being $g$ semi-separated, $h$ is also affine by Lemma \ref{deltafin}. As a consequence, $\Delta_{X|Z}$ is affine as wanted.
With (i) and (ii) proved, the argument \cite[Remarque (\textbf{0}.1.3.9)]{ega1} proves that (iii) and (iv), are equivalent, so let us prove the former. Note that
\((X \times_S S') \times_{(Y \times_S S')} (X \times_S S') \cong 
(X \times_Y X) \times_S S'\) and 
\( \Delta_{X \times_S S'|Y \times_S S'} \cong \Delta_{X|Y} \times_S \id_{S'}\) so
the result follows from \cite[Proposition (9.1.16) (iii)]{ega1}.
To prove (v) set $f \colon X \to Y$ and $g \colon Y \to Z$. Factor $f$ as
\[ X \xto{\Gamma_f} X \times_Z Y \xto{p_2} Y.
\]
Using (iv) we see that $p_2 = (g \circ f) \times_Z \id_Y$ is semi-separated and $\Gamma_f$ is semi-separated by (i), therefore, by (ii), $f$ is semi-separated.

Finally, the diagram
\begin{diagram}[height=2.2em,w=2.2em,p=0.3em,labelstyle=\scriptstyle]
X_\mathrm{red} & \rTo^{f_\mathrm{red}} & Y_\mathrm{red}  \\
\dTo_{j_X}     &                       & \dTo_{j_Y} \\ 
X              & \rTo^{f}              & Y \\  
\end{diagram}
where $j_X$ and $j_Y$ are the canonical embeddings, commutes. By (i), (ii) and (v) it follows that $f_\mathrm{red}$ is semi-separated\footnote{For the first cite to EGA in this proof we could use \cite[Proposition (5.5.1)]{ega2}, for the second \cite[(3.5.1)]{ega160} and for the last \cite[Proposition (1.6.2)]{ega2}}.
\end{proof}

\begin{prop} \label{covss} 
To be semi-separated is local on the base. In other words, let $f \colon X \to Y$ be a morphism of schemes and let $\{ V_{\alpha} \}_{\alpha \in L}$ be an open covering of $Y$, then $f$ is semi-separated if, and only if, its restrictions $f^{-1}(V_{\alpha}) \to  V_{\alpha}$ are semi-separated for all $\alpha \in L$.
\end{prop}

\begin{proof}
The ``only if'' part follows easily from Proposition \ref{sorissep}. For the ``if'' part, let us check that the map $\Delta_{X|Y}$ is affine. Let $U_\alpha := f^{-1}(V_{\alpha})$ for each $\alpha \in L$. Note that $U_{\alpha} \times_Y U_{\alpha} \cong U_{\alpha} \times_{V_{\alpha}} U_{\alpha}$, therefore if we check that $\{ U_{\alpha} \times_{V_{\alpha}} U_{\alpha}\}_{\alpha \in L}$ cover $X \times_Y X$, we are done. And this follows from an argument similar to the last part of the proof of \cite[Proposition (5.3.5)]{ega1}\footnote{\cite[Proposition (5.5.5)]{ega160}}.
\end{proof}

\begin{prop}\label{sskey}
Let $Y$ be an \emph{affine} scheme. Let $\{U_{\alpha}\}_{\alpha \in L}$ be an affine open covering of $X$. A morphism of schemes $f \colon X \to Y$ is semi-separated if, and only if, for any pair of indices $\alpha, \beta \in L$ the open subset $U_{\alpha} \cap U_{\beta}$ is affine.
\end{prop}

\begin{proof}
Note that the open subsets $U_{\alpha} \times_Y U_{\beta}$ constitute an affine covering of $X \times_Y X$ and $\Delta_f^{-1}(U_{\alpha} \times_Y U_{\beta}) = U_{\alpha} \cap U_{\beta}$ which is affine by hypothesis. So, this means that $\Delta_f$ is affine \ie $f$ is semi-separated.
\end{proof}

\begin{cosa}\label{sssch}
In the previous proposition the affine base scheme plays a very limited role, in fact, the characterization is independent of it, so the condition is equivalent to saying that the canonical map $X \to \spec(\ZZ)$ is semi-separated. In this case we say simply that the \emph{scheme} $X$ is \emph{semi-separated}.
\end{cosa}

\begin{rem}
Proposition \ref{sskey}. shows us that being a semi-se\-pa\-ra\-ted scheme in our sense is precisely the same as Thomason and Trobaugh's notion in \cite[Appendix B.7]{tt}. Note that the same applies to the definition of semi-separated morphism.
\end{rem}

\begin{cor}
A scheme $X$ is separated if, and only if, $X$ is semi-separated and given an affine open covering $\{U_{\alpha}\}_{\alpha \in L}$ of $X$ the canonical morphism
\[
\Gamma(U_{\alpha}, \CO_X) \otimes \Gamma(U_{\beta}, \CO_X) \to 
\Gamma(U_{\alpha} \cap U_{\beta}, \CO_X)
\]
is surjective.
\end{cor}

\begin{proof}
This a restatement of \cite[Proposition (5.3.6)]{ega1}\footnote{\cite[Proposition (5.5.6)]{ega160}}. \noqed
\end{proof}

\begin{ex}
Let $K$ be a field. The line ``with its origin doubled'', $X$ \ie the scheme obtained gluing two copies $\AF^1_K$ of along $\AF^1_K \setminus \{\mathbf{0}\}$ using the identity as gluing map is semi-separated. Note that an open subset of $X$ is the complementary subset of a finite number of closed points. The most problematic case arises from removing one of the doubled points. Denote by $\mathbf{0}_1$ and $\mathbf{0}_2$ the two points corresponding to the origin then $(X \setminus \mathbf{0}_1) \cap (X \setminus \mathbf{0}_2)$ is affine and for the rest of couples of affine subsets the fact that their intersection is affine is clear from this. However $X$ is not separated as it is well-known. If we do the same thing with the plane $\AF^2_K$, we obtain a scheme  $X'$ which is not semi-separated. Denote again by $\mathbf{0}_1$ and $\mathbf{0}_2$ the two points corresponding to the origin. The intersection of affine open subsets $(X' \setminus \mathbf{0}_1) \cap (X' \setminus \mathbf{0}_2)$ (both isomorphic to $\AF^2_K$), is $\AF^2_K \setminus \{\mathbf{0}\}$, which is not an affine open subset. Therefore $X'$ is not semi-separated. Note that $X'$ is quasi-separated because it is noetherian, being of finite type over a field. We conclude that the implications
\[   \text{ separated } 
\imp \text{ semi-separated } 
\imp \text{ quasi-separated }
\]
are all strict.
\end{ex}

\begin{cor} 
A morphism $f \colon X \to Y$ is semi-separated if, and only if, for every open subset $V$ of $Y$ such that $V$ is semi-separated, then the open subset $f^{-1}(V)$ of $X$ is also semi-separated.
\end{cor}

\begin{proof}
The direct implication is a consequence of Proposition \ref{sorissep}. The converse follows immediately from Proposition \ref{covss}.
\end{proof}

\begin{rem}
As an example of the importance of the semi-separatedness condition let us mention the following result of Totaro \cite[Proposition 8.1]{to}. For a smooth scheme $X$ over a field the fact that every coherent sheaf on $X$ is a quotient of a vector bundle is equivalent to $X$ being semi-separated. This gives a condition that implies that the natural map $K^\text{naive}_0(X) \to K_0(X)$ in K-theory is an isomorphism. In order to study quasi-coherent sheaves on more general situations like algebraic spaces or algebraic stacks we expect that this condition is the right one to obtain results generalizing those that hold for schemes. 
\end{rem}

\section{Closed structure} \label{closed}

\begin{cosa}
\textbf{Derived tensor product}.
From now on, we will abbreviate and denote simply by $\D(X)$ the category $\D(\A(X))$ and by $\K(X)$ the category $\K(\A(X))$. We recall the definition of tensor products in this category. A complex $\CP^{\bullet}$ is called K-flat if for every acyclic complex $\CA^{\bullet} \in \K(X)$ the complex $\CP^\bullet \otimes^{\LL}_{\CO_X} \CA^\bullet$ is acyclic \cite[Definition 5.1]{S}. Given complexes $\CF^\bullet$ and $\CG^\bullet$ in $\D(X)$, we may compute 
$\CF^\bullet \otimes^{\LL}_{\CO_X} \CG^\bullet$ taking a K-flat resolution \cite[Propostion 6.5]{S} of either $\CF^\bullet$, $\CP^{\bullet}_\CF \to \CF^\bullet$, or of $\CG^\bullet$, $\CP^{\bullet}_\CG \to \CG^\bullet$. In other words
\[
\CP^{\bullet}_\CF \otimes_{\CO_X} \CG^\bullet \liso
\CF^\bullet \otimes^{\LL}_{\CO_X} \CG^\bullet \losi
\CF^\bullet \otimes_{\CO_X} \CP^{\bullet}_\CG.
\]
where the tensor product of complexes is defined as usual (\cfr \cite[(1.5.4)]{LDC}) and the isomorphism is understood in the derived category. This is not trivial because K-flat resolutions are not unique \emph{up to homotopy} ---\ie are not unique in $\K(X)$. See \S2.5, especially (2.5.7), in \lc for a discussion of $\otimes^{\LL}$ in $\D(X)$ (denoted there $\underline{\underline{\otimes}}$)\footnote{See \cite[\S 2.5]{LDC}. In \cite[Definition (2.5.1)]{LDC} the terminology $q$-flat is used for K-flat.}.
\end{cosa}

\begin{cosa} \label{inthomdef}
\textbf{Internal hom}.
There is another essential bifunctor defined on  the category $\D(X)$, namely, $\rshom_X^\bullet$. Given complexes $\CF^\bullet$ and $\CG^\bullet$ in $\D(X)$, we define it deriving the functor $\shom_X^\bullet$ (defined, for instance in \cite[(2.4.5)]{LDC}). To compute it, fix a K-injective resolution $\CG^\bullet \to \CI^{\bullet}_\CG$. So we have
\[
\rshom_X^\bullet(\CF^\bullet, \CG^\bullet) =
 \shom_X^\bullet(\CF, \CI^{\bullet}_\CG).
\]

\end{cosa}

\begin{rem}
The existence of K-injective resolutions for sheaves of modules on a ringed space is due to Spalstenstein \cite{S}. The reader will find a proof valid for any Grothedieck category (in particular for $\A_\qc(X)$ for any scheme $X$) in \cite[Theorem 5.4]{AJS}.
\end{rem}

\begin{lem}
On a quasi-compact and \emph{semi-separated} scheme $X$, every $\CF \in \D(\A_\qc(X))$ has a K-flat resolution (in $\K(X)$) $\CP^{\bullet}_\CF \to \CF^\bullet$ made up of quasi-coherent sheaves.
\end{lem}

\begin{proof}
The lemma is a very slight generalization of \cite[Proposition (1.1)]{AJL1}. We have substituted the hypothesis of separated by semi-separated. In the proof  of the cited result it is only used the fact that a finite intersection of affine open subsets of $X$ is affine, as for instance, to assert that the maps $\lambda_i$ (Remark (b) at the end of page 11 in \lc\negthickspace) are affine. But this holds precisely in the semi-separated case as follows from Proposition \ref{sskey}.
\end{proof}

\begin{cosa} Using the previous lemma we are able to define a bifunctor
\[
-\otimes^{\LL}_{\CO_X}\!\!\!- \colon 
\D(\A_\qc(X)) \times \D(\A_\qc(X)) \to \D(\A_\qc(X))
\]
whenever $X$ is a quasi-compact and semi-separated scheme. Indeed, the tensor product of quasi-coherent modules is again quasi-coherent by \cite[Corollaire (1.3.12), (i)]{ega1}\footnote{\cite[Corollaire (1.3.12), (i)]{ega160}}. Therefore if $\CF, \CG \in \D(\A_\qc(X))$ we have
\(
\CF^\bullet \otimes^{\LL}_{\CO_X} \CG^\bullet \cong
\CP^{\bullet}_\CF \otimes_{\CO_X} \CG^\bullet
\) and this last complex is made of quasi-coherent sheaves.
\end{cosa}

\begin{cosa} \label{qctor}
\textbf{The quasi-coherator}.
Recall briefly that, for a quasi-compact and quasi-separated scheme $X$, the canonical inclusion $i \colon \A_\qc(X) \to \A(X)$ has a right adjoint $Q_X \colon \A(X) \to \A_\qc(X)$. The functor $i$ is exact so it induces a $\Delta$-functor $\mathbf{i} \colon \D(\A_\qc(X)) \to \D(X)$ with right adjoint $\R{}Q_X \colon \D(X) \to \D(\A_\qc(X))$. We denote this adjunction by $\mathbf{i} \dashv \R{}Q_X$. Most of the time one omits writing the functor $\mathbf{i}$, leaving it implicit. The definition of $Q_X$ goes back to \cite[p. 187, Lemme 3.2]{I}. One can find its construction also in \cite[B.12, Lemma]{tt}. This functor is sometimes called the \emph{quasi-coherator}. It is clear that the essential image of $\mathbf{i}$ is contained in $\D_\qc(X)$ where by this we denote the full subcategory of $\D(X)$ formed by the complexes whose homologies are quasi-coherent sheaves. If $X$ is a quasi-compact and semi-separated scheme then the induced functor $\D(\A_\qc(X)) \to \D_\qc(X)$ is an equivalence of categories with quasi-inverse induced by $\R{}Q_X$ as is provided by \cite[Corollary 5.5]{BN} or \cite[Proposition (1.3)]{AJL1}. In both references the separated hypothesis can be weakened to semi-separated. 
\end{cosa}

\begin{cosa}
One must be careful with the functor $Q_X$ and its derived functor $\R{}Q_X$, they \emph{do not} commute with localization. If $j \colon U \inc X$ is the canonical inclusion of an open subset $U$ into its ambient scheme $X$, then for $\CF^\bullet \in \D(X)$ the canonical map $j^* \R{}Q_X \CF^\bullet \to \R{}Q_U j^*\CF^\bullet$ need not be an isomorphism, even in simple cases. 

Let us illustrate this fact by an example. First of all note that the functor $Q_X$ being a right adjoint is left exact and therefore $R^0Q_X \CF = Q_X \CF$ with $\CF \in \A(X)$, therefore it is enough to treat the case of $Q_X$. Now consider $X = \spec(R)$ with $R$ a discrete valuation ring. This affine scheme has two points, the generic point $\xi$ corresponding to the zero ideal, and the closed point $x$ corresponding to the maximal ideal. There is just a nonempty open subset different from $X$, $U =\{\xi\}$. To give a sheaf of $\widetilde{R}$-Modules $\CF$ is the same as to give a $R$-linear map $M \to L$ where $M$ is a $R$-module and $L$ is a $K(R)$-vector space where $K(R)$ denotes the field of fractions of $R$. Note that $M = \CF(X)$ and $L = \CF(U)$. The sheaf $\CF$ is quasi-coherent precisely when $M \otimes_R K(R)Ê\cong L$ and the map $M \to L$ is the canonical restriction map. Note that $Q_X \CF$ is the sheaf whose associated map is $M \to M \otimes_R K(R)$ and the map $Q_X \CF \to \CF$ is the obvious one. Take any $\CF$ that is not quasi-coherent. Then the sheaf $Q_U j^*\CF$ corresponds to the $K(R)$-vector space $L$ while $j^*Q_X \CF$ corresponds to $M \otimes_R K(R)$ and the natural map $M \otimes_R K(R)Ê\to L$ is not an isomorphism by assumption. (Note that here Enochs and Estrada's description of sheaves as module representations of the quiver ring $R \to K(R)$ is \emph{isomorphic} to the category of sheaves over $\spec(R)$.)
\end{cosa}

\begin{cosa} \label{inthomdefqc}
\textbf{The internal hom}.
Let $X$ be a quasi-compact and semi-separated scheme. The functor $Q_X$ allows us to introduce an internal hom-functor (or function space functor) in $\A_\qc(X)$ and, deriving it, in $\D(\A_\qc(X))$. It is well-known that for $\CF, \CG \in \A_\qc(X) \subset \A(X)$ it is not guaranteed that $\shom_X(\CF, \CG) \in \A_\qc(X)$. However,  $Q_X\!\shom_X(\CF, \CG)$ is obviously a quasi-coherent sheaf. With this in mind, we define
\[
\dhom^\bullet_X(\CF^\bullet, \CG^\bullet) := 
\R{}Q_X\!\rshom^\bullet_X(\CF^\bullet, \CG^\bullet)
\]
for $\CF^\bullet, \CG^\bullet \in \D(\A_\qc(X))$. In what follows we will often argue by localizing to small open subsets. The definition of $\dhom$ makes it clear that does not localize by the previous discussion on $\R{}Q_X$. In every instance that we need to use a localization argument we will refer the property we want to show to a property of $\shom_X$ that is compatible with localizations (by its definition).
\end{cosa}


\begin{prop}\label{inthom37}
Let $X$ be a quasi-compact and semi-separated scheme. Let $\CF, \CG, \CH \in \D(\A_\qc(X))$ we have a natural isomorphism in $\D(\A_\qc(X))$
\[
 \Hom_{\D(\A_\qc(X))}(\CF, \dhom^\bullet_X(\CG, \CH)) \iso
 \Hom_{\D(\A_\qc(X))}(\CF \otimes^{\LL}_{\CO_X} \CG, \CH).
\]
In other words, we have an adjunction $-\otimes^{\LL}_{\CO_X}\CG \dashv \dhom^\bullet_X(\CG, -)$ in the category  $\D(\A_\qc(X))$.
\end{prop}

\begin{proof}
Consider the following chain of isomorphisms:
\begin{align*}
\Hom_{\D(\A_\qc(X))}(\CF^\bullet, &\dhom^\bullet_X(\CG^\bullet, 
            \CH^\bullet)) \cong \\
    & \cong \Hom_{\D(X)}(\CF^\bullet, \rshom^\bullet_X(\CG^\bullet, \CH^\bullet)) 
                           \tag{$\mathbf{i} \dashv \R{}Q_X$}\\                            
    & \cong \Hom_{\D(X)}(\CF^\bullet \otimes^{\LL}_{\CO_X} \CG^\bullet, 
            \CH^\bullet) 
                          \tag{\cite[(3.5.2), (d)]{LDC}} \\
    & \cong \Hom_{\D(\A_\qc(X))}(\CF^\bullet \otimes^{\LL}_{\CO_X} \CG^\bullet, 
            \CH^\bullet).
\end{align*}
For the first and last isomorphisms note that the canonical embedding functor $\mathbf{i} \colon \D(\A_\qc(X)) \to \D(X)$ is full because it factors through the equivalence of categories $\D(\A_\qc(X)) \to \D_\qc(X)$ in \ref{qctor} and the full embedding $\D_\qc(X) \inc \D(X)$. 
The composed isomorphism is the one we were looking for.
\end{proof}

\begin{rem}
We also have that the adjunction $-\otimes^{\LL}_{\CO_X}\CG \dashv \dhom^\bullet_X(\CG, -)$ holds internally in $\D(\A_\qc(X))$ for every $\CG \in \D(\A_\qc(X))$, \ie we have the following canonical isomorphism for all $\CF$ and $\CH$ in $\D(\A_\qc(X))$ 
\[
 \dhom^\bullet_X(\CF, \dhom^\bullet_X(\CG, \CH)) \iso
 \dhom^\bullet_X(\CF \otimes^{\LL}_{\CO_X} \CG, \CH).
\]
This follows formally from the axioms of closed category, see \cite[Exercise (3.5.3) (e)]{LDC}.
\end{rem}

\begin{thm} \label{asht2}
Let $X$ be a quasi-compact semi-separated scheme. The category $\D(\A_\qc(X))$ has a natural structure of symmetric closed category. In other words, axiom (\ref{scl}) of \ref{axioms} holds.
\end{thm}

\begin{proof}
First, the data $(\D(\A_\qc(X)), \otimes^{\LL}_{\CO_X}, \CO_X)$ together with the compatibility diagrams correspond to a monoidal category. Indeed, by \cite[Examples (3.5.2) (d)]{LDC} this is the case for $(\D(X), \otimes^{\LL}_{\CO_X}, \CO_X)$. But the category $\D(\A_\qc(X))$ is equivalent to $\D_\qc(X)$, that is a full subcategory of $\D(X)$. Once we know that the bifunctor $\otimes^{\LL}_{\CO_X}$ restricts to $\D(\A_\qc(X))$, all the diagrams corresponding to the coherence axioms belong to this subcategory. The adjunction $-\otimes^{\LL}_{\CO_X}\CG \dashv \dhom^\bullet_X(\CG, -)$ for every $\CG \in \D(\A_\qc(X))$, holds by Proposition \ref{inthom37}. It is clear that both bifunctors are $\Delta$-functors in either variable. Finally, consider the diagram
\begin{diagram}[height=2em,w=2em,p=0.3em,labelstyle=\scriptstyle]
\CO_X[r] \otimes^{\LL}_{\CO_X} \CO_X[s] & \rTo^{\sim}  &  \CO_X[r + s] \\
\dTo_\theta                             &              & \dTo_{(-1)^{rs}} \\ 
\CO_X[s] \otimes^{\LL}_{\CO_X} \CO_X[r] & \rTo^{\sim}  &  \CO_X[r + s] \\
\end{diagram}
where $\theta$ is defined as in \cite[(1.5.4.1)]{LDC}. With the sign introduced in this definition it is clear that the square is commutative,  having in mind that $\CO_X$ is K-flat considered as a complex concentrated in degree 0. Note that $\theta$ corresponds to $T$ in \cite[Definition A.2.1(4)]{hps}.
\end{proof}

\section{Strongly dualizable generators}

\begin{cosa}\label{compact}
Let $\T$ be a triangulated category. An object $E$ of $\T$ is called \emph{compact} if the functor $\Hom_{\T}(E,-)$ commutes with arbitrary coproducts. Another way of expressing this condition is that a map from $E$ to a coproduct factors through a finite subcoproduct.
\end{cosa}

\begin{prop} \label{cgenss}
Let $X$ be a quasi-compact semi-separated scheme. The category $\D(\A_\qc(X))$ is generated (in the sense of \ref{genadd}) by compact objects.
\end{prop}

\begin{proof}
This is a generalization of \cite[Proposition 2.5]{Ngd}, where it is stated for a quasi-compact, \emph{separated} scheme. The proof of \lc is based on Lemma 2.6 there. Note that the only property of separated schemes that it is used in the proof of the Lemma is that the intersection of open affine subschemes is affine, something that holds when $X$ is just semi-separated by Proposition \ref{sskey}.
\end{proof}

\begin{cosa}
A complex $\CF^\bullet \in \D(\A_\qc(X))$ is called \emph{strongly dualizable} if, and only if, the canonical map
\[
\dhom^\bullet_X(\CF^\bullet, \CO_X) \otimes^{\LL}_{\CO_X} \CG^\bullet \lto
\dhom^\bullet_X(\CF^\bullet, \CG^\bullet)
\]
is an isomorphism for all $\CG^\bullet \in \D(\A_\qc(X))$, \cite[Definition 1.1.2]{hps}.

A complex $\CE^\bullet \in \CCC(X)$ is called \emph{perfect} if for every $x \in X$ there is an open neighborhood $x \in U \subset X$ such that, denoting $j \colon U \inc X$ the canonical inclusion, the complex $j^*(\CE^\bullet)$ is quasi-isomophic to a bounded complex made up of locally free finite-type Modules over $U$.
\end{cosa}

\begin{prop} \label{perfsd}
Let $X$ be a quasi-compact semi-separated scheme. A perfect complex in $\D(\A_\qc(X))$ is strongly dualizable.
\end{prop}

\begin{proof}
Let $\CE^\bullet$ be a perfect complex and $\CG^\bullet \in \D(\A_\qc(X))$. Choose a K-injective resolution $\CG^\bullet \to \CI^{\bullet}_\CG$ in such a way that \[\rshom^\bullet_X(\CE^\bullet, \CG^\bullet) = \shom^\bullet_X(\CE^\bullet, \CI^{\bullet}_\CG).\] Being $\CE^\bullet$ a perfect complex and $\CI^{\bullet}_\CG$ of quasi-coherent homology, it follows that $\shom^\bullet_X(\CE^\bullet, \CI^{\bullet}_\CG)$ has quasi-coherent homology too, by \cite[Theorem 2.4.1]{tt}. Now, using the equivalence of categories $\D(\A_\qc(X)) \iso \D_\qc(X)$ from \ref{qctor} we have that $\dhom^\bullet_X(\CE^\bullet, \CG^\bullet) \cong \rshom^\bullet_X(\CE^\bullet, \CG^{\bullet})$. The same argument applies to $\dhom^\bullet_X(\CE^\bullet, \CO_X)$, so we are reduced to prove that
\[
\rshom^\bullet_X(\CE^\bullet, \CO_X) \otimes^{\LL}_{\CO_X} \CG^\bullet \lto
\rshom^\bullet_X(\CE^\bullet, \CG^\bullet)
\]
is an isomorphism in $\D(X)$. After taking appropriate resolutions, we see that this is a local problem. Restricting to a small enough open set we can assume that each $\CE^i$ are free modules of finite type for every $i \in \ZZ$ and in this case the fact that the map is an isomorphism is clear.
\end{proof}

\begin{cosa} \label{strong}
In \ref{genadd} we have defined what it means for an object of an additive category to be a generator. Note, however that the notion used in the definition of stable homotopy theory is stronger (see \cite[Definition 1.1.4 (c)]{hps}). To distinguish both notions, we will say, for a triangulated category $\T$ in which all coproducts exist, that a set of objects $\CS$ \emph{generates $\T$ in the strong sense} if the smallest triangulated subcategory closed for coproducts that contains $\CS$ is all of $\T$, in accordance with \ref{axioms} (\ref{sdg}) (=\cite[1.1.4 (d)]{hps}). These two notions agree, for the cases we are interested in, by \cite[Lemma 3.2]{Ngd} which says that $\T$ is a triangulated category generated by a set of compact objects $\{T_\lambda \, / \, \lambda \in \Lambda\}$ containing all translations (suspensions) of its members if and only if this set generates $\T$ in the strong sense.
\end{cosa}

\begin{lem}\label{mj} 
Let $X$ be a quasi-compact semi-separated scheme and let $U$ be an affine open subset of $X$ and denote by $j \colon U \inc X$, the canonical inclusion. If $\CE^\bullet$ is a compact object in $\D(\A_\qc(X))$ then its restriction, $j^*\CE^\bullet$, is a compact object in $\D(\A_\qc(U))$.
\end{lem}

\begin{proof}
$X$ is semi-separated and $U$ affine implies that $j$ is an affine map. As a consequence the functor $j_*$ is exact and induces a functor in the derived category without need to derive it. Note that $j_*$ preserves quasi-coherence. Let $\{\CF^{\bullet}_\lambda \, / \,\lambda \in \Lambda\}$ be a set of objects in $\D(\A_\qc(U))$. The canonical morphism $\phi \colon \oplus_{\lambda \in \Lambda} j_*\CF^{\bullet}_\lambda \,\iso\, j_*\oplus_{\lambda \in \Lambda}\CF^{\bullet}_\lambda$ is an isomorphism. Therefore we have the following
\begin{align*}
\oplus_{\lambda \in \Lambda}\Hom&_{\D(\A_\qc(U))}(j^*\CE^\bullet, \CF^{\bullet}_\lambda)
 \cong \\
    & \cong \oplus_{\lambda \in \Lambda}\Hom_{\D(\A_\qc(X))}(\CE^{\bullet}, 
            j_*\CF^{\bullet}_\lambda)
                       \tag{$j^* \dashv j_*$}\\
    & \cong \Hom_{\D(\A_\qc(X))}(\CE^{\bullet}, 
            \oplus_{\lambda \in \Lambda}j_*\CF^{\bullet}_\lambda) 
                       \tag{$\CE^{\bullet}$ compact}\\
    & \cong \Hom_{\D(\A_\qc(X))}(\CE^{\bullet}, 
            j_*\oplus_{\lambda \in \Lambda}\CF^{\bullet}_\lambda) 
                       \tag{$\phi$ isomorphism}\\
    & \cong \Hom_{\D(\A_\qc(U))}(j^*\CE^\bullet, 
            \oplus_{\lambda \in \Lambda}\CF^{\bullet}_\lambda) 
                       \tag{$j^* \dashv  j_*$}.
\end{align*}
So we conclude that $j^*\CE^\bullet$ is a compact object.
\end{proof}

\begin{prop} \label{compsd}
Let $X$ be a quasi-compact semi-separated scheme. An object in $\D(\A_\qc(X))$ is compact if and only if it is a perfect complex.
\end{prop}

\begin{proof}
Suppose first that $X$ is affine, namely $X = \spec(A)$. By \cite[Lemma 4.3]{AJST}, in $\D(A\md) \cong \D(\A_\qc(X))$ the compact objects are just the perfect ones.

Now, for general quasi-compact semi-separated $X$, let $\CE^\bullet$ be a compact object. Let $U$ be an affine open subset of $X$ and denote by $j \colon U \inc X$, the canonical inclusion. By Lemma \ref{mj}, $j^*\CE^\bullet$, is a compact object and therefore, by the previous discussion it is perfect. But to be perfect is a local condition, therefore $\CE^\bullet$ is perfect.

Let us see now that perfect implies compact. Let $\CE^\bullet$ be a perfect complex and $\{\CF^\bullet_\lambda \in \D(\A_\qc(X)) \,/\, \lambda \in \Lambda\}$. Consider the canonical map 
\[
\phi \colon
\oplus_{\lambda \in \Lambda} \rshom_X(\CE^\bullet, \CF^\bullet_\lambda) \lto
\rshom_X(\CE^\bullet, \oplus_{\lambda \in \Lambda} \CF^\bullet_\lambda)
\]
We want to show first that $\phi$ is an isomorphism in $\D(X)$, a local question, therefore we may take a point $x \in X$ and an open neighborhood $V \subset X$ of $x$ such that $\CE^\bullet |_V$ is a bounded complex of free finite type modules. Replace $X$ by $V$. If the complex $\CE^\bullet$ has length one then it is trivial. If the complex has length $n > 1$, suppose that $\CE^q$ is the first nonzero object for a certain $q \in \ZZ$. Then there is a distinguished triangle $\CE^q[-q] \to \CE^\bullet \to {\CE'}^\bullet \overset{+}\to$ with ${\CE'}^\bullet$ of length $n - 1$. The isomorphism holds for $\CE^q[-q]$ and  for ${\CE'}^\bullet$ by induction, therefore it has to hold for $\CE^\bullet$. Now, arguing as in \cite[Theorem 2.4.1]{tt} we have that 
$\rshom_X(\CE^\bullet, \CF^\bullet_\lambda) \in \D_\qc(X)$. We have the following chain of canonical isomorphisms
\begin{align*}
\oplus_{\lambda \in \Lambda} &\Hom_{\D(\A_\qc(X))}(\CE^\bullet, \CF^\bullet_\lambda) \cong \\
    & \cong \h^0(\oplus_{\lambda \in \Lambda} \rhom^\bullet_X(\CE^\bullet, \CF^\bullet_\lambda)) 
                   \tag{$\h^0$ commutes with $\oplus$}\\
    & \cong \h^0(\oplus_{\lambda \in \Lambda} \R\Gamma(X, \rshom^\bullet_X(\CE^\bullet, \CF^\bullet_\lambda))) \\
    & \cong \h^0(\R\Gamma(X, \oplus_{\lambda \in \Lambda}\rshom^\bullet_X(\CE^\bullet, \CF^\bullet_\lambda))) 
                   \tag{by \cite[Corollary (3.9.3.3)]{LDC}}\\
    & \cong \h^0(\R\Gamma(X, \rshom^\bullet_X(\CE^\bullet, \oplus_{\lambda \in \Lambda}\CF^\bullet_\lambda)))
                   \tag{via $\phi$}\\
    & \cong \Hom_{\D(\A_\qc(X))}(\CE^\bullet, \oplus_{\lambda \in \Lambda} \CF^\bullet_\lambda). 
\end{align*}
And this shows that $\CE^\bullet$ is compact in $\D(\A_\qc(X))$.
\end{proof}

\begin{rem}
The fact that for a quasi-compact and \emph{separated} scheme $X$, a perfect complex is a compact object in $\D(\A_\qc(X))$ and reciprocally is already stated in \cite[Example 1.13 and Corollary 2.3]{Ngd}.
\end{rem}

\begin{thm}\label{asht3} 
Let $X$ be a quasi-compact semi-separated scheme. Axiom (\ref{sdg}) holds in the category $\D(\A_\qc(X))$. 
\end{thm}

\begin{proof}
The proof is contained in the previous discussion. Indeed, by Proposition \ref{cgenss}, $\D(\A_\qc(X))$ is compactly generated. By the discussion in \ref{strong} this means that there exists a set of compact objects such that the smallest triangulated subcategory of $\D(\A_\qc(X))$ closed for coproducts is all of $\D(\A_\qc(X))$. Proposition \ref{compsd} tells us that the compact generators are perfect complexes. But perfect complexes are strongly dualizable by Proposition \ref{perfsd}, so we have completed the proof.
\end{proof}

\begin{cor} \label{main}
Let $X$ be a quasi-compact semi-separated scheme. The category $\D(\A_\qc(X))$ is a stable homotopy category in the sense of \cite{hps}. 
\end{cor}

\begin{proof}
Combine Theorems \ref{asht145}, \ref{asht2} and \ref{asht3}.
\end{proof}

\begin{cor}
Let $X$ be a quasi-compact semi-separated scheme. The category $\D(\A_\qc(X))$ is an \emph{unital algebraic} stable homotopy category in the sense of \cite{hps}. 
\end{cor}

\begin{proof}
The adjective ``algebraic'' just means that the set of generators is formed by compact objects \cite[after Definition 1.1.4]{hps} and this holds by Proposition \ref{compsd}. And ``unital'' means that the unit for the tensor product bifunctor, $\CO_X$, is compact, which is clear.
\end{proof}

\begin{rem}
Some readers may wonder if given a scheme $X$, replacing the category $\D(\A_\qc(X))$ by $\D_\qc(X)$, the semi-separation hypothesis could be relaxed. It does not seem so, at least for property (ii) ---being a closed category. One should invoke the equivalence mentioned in \ref{qctor}, and again, the semi-separated hypothesis is present. Note that the equivalence implies that in this case $\D_\qc(X)$ is also a stable homotopy theory.
\end{rem}

\begin{rem}
By \cite[Theorem 3.1.1.(2)]{bb} (or \cite[Theorem 4.2.]{LN} with a different construction in the separated case), the category $\D(\A_\qc(X))$ is not only algebraic, but even \emph{monogenic} (\ie generated by a single compact object). Note that in \cite{LN} it is used systematically the equivalence between $\D(\A_\qc(X))$ and $\D_\qc(X)$.
\end{rem}

\section{On the closed structure of $\D(\A_\qc(X))$}

In this section we will discuss a rather technical question, the relationship between the closed structures in $\A_\qc(X)$ and $\D(\A_\qc(X))$ where $X$ denotes a (quasi-compact and semi-separated) scheme. It is independent of the rest of the discussion in the present paper, so it can be skipped if the reader is not interested. In short, the definition in \ref{inthomdefqc} of the internal-hom functor in $\D(\A_\qc(X))$ is \emph{not} as the derived internal-hom functor in $\A_\qc(X)$, as a matter of fact it is defined as a certain composition of functors. This definition is convenient for our proofs. We will show however that, whenever $X$ is quasi-compact and semi-separated, this can be interpreted as the derived functor of the internal hom defined for $\A_\qc(X)$ as long as we derive in the sense of quasi-coherent sheaves \ie with injective quasi-coherent resolutions.

For a scheme $X$, recalling the definitions of \ref{qctor} and abbreviating as usual $\D(X)$ and $\K(X)$ for $\D(\A(X))$ and $\K(\A(X))$, respectively, we have a commutative diagram
\begin{diagram}[height=2.2em,w=3.3em,labelstyle=\scriptstyle]
\K(\A_\qc(X)) & \rTo^{i}          & \K(X) \\
\dTo_{q_\qc}  &                   & \dTo_{q} \\ 
\D(\A_\qc(X)) & \rTo^{\,\mathbf{i}} & \D(X)  \\  
\end{diagram}
where the map $i \colon \K(\A_\qc(X)) \to \K(X)$ is induced from the map $i \colon \A_\qc(X) \to \A(X)$ by additivity and the functors $q_\qc$ and $q$ are the canonical ones from a homotopy category to its derived category.

Let $X$ be a scheme. Let $\CG^\bullet \in \D(\A_\qc(X))$. Let us denote by $\R_\qc(Q_X \circ \shom^\bullet_X(\CG^\bullet,-))$ the derived functor of the functor
\[
Q_X \circ \shom^\bullet_X(\CG^\bullet,-) \colon \K(\A_\qc(X)) \to \D(\A_\qc(X))
\]
\ie for a complex $\CF^\bullet \in \K(\A_\qc(X))$, take a K-injective resolution \emph{in} $\A_\qc(X)$, $\CF^\bullet \to \CI^\bullet_\CF$, then
\[ \R_\qc(Q_X \circ \shom^\bullet_X(\CG^\bullet,\CF^\bullet)) =
Q_X (\shom^\bullet_X(\CG^\bullet,\CI^\bullet_\CF)).
\]

\begin{rem}
The reader should be aware that, in general, it can be very different to take derived functors in the sense of $\D(\A_\qc(X))$ or of $\D(X)$. We will use the notation $\R_\qc$ for right derived functors in $\D(\A_\qc(X))$. For an iluminating discussion and further references \cfr \cite[B.4]{tt}.
\end{rem}

\begin{prop}
Let $X$ be a quasi-compact semi-separated scheme. Let $\CG^\bullet \in \D(\A_\qc(X))$. There is natural tranformation of $\Delta$-functors
\[
\phi_{\CG^\bullet} = \phi \colon 
\R_\qc(Q_X \circ \shom^\bullet_X(\CG^\bullet,-)) \lto
\dhom^\bullet_X(\CG^\bullet,-).
\]
\end{prop}

\begin{proof}
Here and for the sake of simplicity we leave the functor $\mathbf{i}$ implicit in the formulas.
There are canonical natural transformations
\[
\xi \colon q \circ \shom^\bullet_X(\CG^\bullet,-) \lto 
\R{}\shom^\bullet_X(\CG^\bullet,-) \circ q_\qc
\]
and
\[
\xi' \colon q_\qc \circ Q_X \lto 
\R{}Q_X \circ q.
\]
Together they induce the natural transformation
\[
\xi'' \colon q_\qc \circ Q_X \circ \shom^\bullet_X(\CG^\bullet,-) \lto 
\R{}Q_X \circ \R{}\shom^\bullet_X(\CG^\bullet,-) \circ q_\qc
\]
but by the 2-universal property of derived functors this should factor through the universal map
\[
\eta \colon q_\qc \circ Q_X \circ \shom^\bullet_X(\CG^\bullet,-) \lto 
\R_\qc(Q_X \circ \shom^\bullet_X(\CG^\bullet,-)) \circ q_\qc
\]
by a natural transformation
\[
\phi \colon \R_\qc(Q_X \circ \shom^\bullet_X(\CG^\bullet,-))
\lto
\R{}Q_X \circ \R{}\shom^\bullet_X(\CG^\bullet,-)
\]
which is our desired map.
\end{proof}

\begin{cosa} \label{quienesson}
Let $\CK^\bullet \in \K(\A_\qc(X))$ and $\CK^\bullet \to \CI^\bullet_\CK$
be a K-injective resolution \emph{in} $\A(X)$, we have a commutative diagram
\begin{diagram}[height=2em,w=2em,labelstyle=\scriptstyle]
\CK^\bullet &                 & \rTo^{\alpha}      & &\CI^\bullet_\CK \\
            & \rdTo_{\alpha'} &                    & \ruTo_{\gamma} & \\
            &                 & Q_X\CI^\bullet_\CK & &\\  
\end{diagram}
where $\gamma$ is the counit of the adjunction $i \dashv Q_X$ and $\alpha'$ is the map adjoint to $\alpha$. Note that we omit the functor $i$ for simplicity. Now, $\alpha$ is a quasi-isomorphism and so is $\gamma$ because $\CI^\bullet_\CK$ has quasi-coherent homology \cite[Proposition (1.3)]{AJL1}, therefore, by the commutativity, $\alpha'$ is a quasi-isomorphism, too. It is clear that $Q_X\CI^\bullet_\CK$ is K-injective in $\K(\A_\qc(X))$. Indeed, let $\CA^\bullet$ be an acyclic object of $\K(\A_\qc(X))$, then
\[
\Hom_{\K(\A_\qc(X))}(\CA^\bullet, Q_X\CI^\bullet_\CK) \cong 
\Hom_{\K(X)}(i\CA^\bullet, \CI^\bullet_\CK) = 0
\]
the last equality due to the fact that the functor $i$ is exact ---therefore $i\CA^\bullet$ is acyclic--- and that $\CI^\bullet_\CK$ is K-injective.

There is a commutative diagram
\begin{diagram}[height=2em,w=2em,labelstyle=\scriptstyle]
Q_X(\shom^\bullet_X(\CG^\bullet,\CK^\bullet)) & & \rTo^{\xi''} & & 
\dhom^\bullet_X(\CG^\bullet,\CK^\bullet) \\
       & \rdTo(1,2)_{\eta} &                    & \ruTo(1,2)_{\phi} & \\ 
       &    & \R_\qc(Q_X \shom^\bullet_X(\CG^\bullet,\CK^\bullet)) & &\\  
\end{diagram}
that may be described as follows. The map $\eta$ is obtained applying the functor $Q_X \circ \shom^\bullet_X(\CG^\bullet,-)$ to $\alpha'$. To make $\xi''$ explicit, we first apply the functor $\shom^\bullet_X(\CG^\bullet,-)$ to $\alpha$ and obtain 
\[
Q_X(\xi) \colon Q_X(\shom^\bullet_X(\CG^\bullet,\CK^\bullet)) \to 
Q_X(\R{}\shom^\bullet_X(\CG^\bullet,\CK^\bullet))
\]
then we let $\CH^\bullet := \R{}\shom^\bullet_X(\CG^\bullet,\CK^\bullet)$ and $\alpha'' \colon \CH^\bullet \to \CI^\bullet_\CH$ a K-injective resolution. Applying $Q_X$ to this map we obtain \(\xi' \colon Q_X(\CH^\bullet) \to \R{}Q_X(\CH^\bullet)\). Now our desired description is $\xi'' = \xi' \circ Q_X(\xi)$.
Note that $\phi$ is the unique map that makes the diagram commute.
\end{cosa}

\begin{lem} \label{agaf}
Let $X = \spec(A)$ be an affine scheme. Then, the natural transformation $\phi_{X,\CG^\bullet}$ is an isomorphism for any $\CG^\bullet \in \D(\A_\qc(X)) \cong \D(A)$.
\end{lem}

\begin{proof}
Let $\CK^\bullet \in \K(\A_\qc(X))$ as in \ref{quienesson} and $\CK^\bullet \to \CI^\bullet_\CK$ be a K-injective resolution in $\K(X)$. Note first that $Q_X(-) = \widetilde{\Gamma(X, -)}$ and being $X$ affine it is \emph{exact} whenever it is restricted to quasi-coherent sheaves. Take $P^\bullet \to \Gamma(X, \CG^\bullet)$ a K-projective resolution in $\K(A)$ and let $\CP^\bullet := \widetilde{P^\bullet}$. The map $\phi$ in the last diagram from \ref{quienesson} is identified with the composition of isomorphisms:
\[
\widetilde{\Hom^\bullet_X(\CG^\bullet, Q_X\CI^\bullet_\CK)} \,\liso\,
\widetilde{\Hom^\bullet_X(\CG^\bullet, \CI^\bullet_\CK)}    \,\liso\, 
\widetilde{\Hom^\bullet_X(\CP^\bullet, \CI^\bullet_\CK)}
\]
thus $\phi$ is an isomorphism in $\D(\A_\qc(X))$.
\end{proof}
\begin{lem} \label{agfst}
Let $X$ be a quasi-compact semi-separated scheme, $U$ a quasi-compact open subset of $X$ and $u \colon U \inc X$ the canonical embedding. There is a natural isomorphism $\mathbf{i}_X \circ \R_\qc{}u_{*} \,\iso\, \R{}u_{*} \circ \mathbf{i}_U$. (Notation as in the beginning of the section).
\end{lem}

\begin{proof}
By \cite[Proposition (1.3)]{AJL1} it is enough to show that $\R_\qc{}u_{*} \,\iso\, \R{}Q_X \circ \R{}u_{*} \circ \mathbf{i}_U$. Let $\CK^\bullet \in \K(\A_\qc(U))$ and let $\mathbf{i}_U \CK^\bullet \to \CI^\bullet$ be a K-injective resolution (in $\K(U)$). The induced map $\CK^\bullet \to Q_U\CI^\bullet$ is a K-injective resolution in $\K(\A_\qc(U))$. We have the following isomorphisms
\begin{align*}
\R_\qc{}u_{*}\CK^\bullet & \cong u_{*}Q_U\CI^\bullet  \\
    & \cong Q_X u_{*}\CI^\bullet 
               \tag{\cite[Appendix B.13]{tt}}\\
    & \cong \R{}Q_X (u_{*}\CI^\bullet)
               \tag{$u_*\CI^\bullet$ K-injective}\\
    & \cong \R{}Q_X \R{}u_{*} \mathbf{i}_U\CK^\bullet .
               \tag{$\CI^\bullet$ is a resolution of $\mathbf{i}_U\CK^\bullet$}
\end{align*}
and our result follows. 
\end{proof}


\begin{thm}
Let $X$ be a quasi-compact semi-separated scheme. Then, the natural transformation $\phi_{\CG^\bullet}$ is an isomorphism for any $\CG^\bullet \in \D(\A_\qc(X))$.
\end{thm}

\begin{proof}
Again we omit for simplicity the functors $i$ and $\mathbf{i}$.
Let $U$ be a quasi-compact open subset of $X$ and $u \colon U \inc X$ the canonical embedding. We will see that for every $\CG^\bullet \in \D(\A_\qc(X))$ the morphism $\phi_{\CG^\bullet} \circ \R_\qc{}u_{*}$ is an isomorphism. The theorem is the special case in which $X = U$. Let $s(U)$ denote the smallest number of open affines that are needed to cover $U$, we will argue by induction on $s(U)$.

If $s(U) = 1$ then $U$ is an affine open subset of $X$ and being $X$ semi-separated, the functor $u_* \colon \A_\qc(U) \to \A_\qc(X)$ is exact, therefore
\[
\R_\qc(Q_X\shom^\bullet_X(\CG^\bullet, \R_\qc u_* \CK^\bullet))
\cong
\R_\qc(Q_X\shom^\bullet_X(\CG^\bullet, u_* Q_U \CI^\bullet))
\]
where $\CK^\bullet \to \CI^\bullet$ is a K-injective resolution, which implies that $\CK^\bullet \to Q_U \CI^\bullet$ is quasi-coherent K-injective resolution in $\K(\A_\qc(X))$. In this case it holds that $\phi_{\CG^\bullet} \circ \R_\qc{}u_{*} = \phi_{\CG^\bullet} \circ u_{*}$ and agrees with the composition of the following chain of isomorphisms:
\begin{align*}
\R_\qc(Q_X&\shom^\bullet_X(\CG^\bullet, u_* Q_U \CI^\bullet)) \cong \\
    & \cong Q_Xu_*\shom^\bullet_U(u^*\CG^\bullet, Q_U \CI^\bullet) 
               \tag{$u_* Q_U \CI^\bullet$ is K-inj. in $\K(\A_\qc(U))$}\\
    & \cong u_* \R_\qc(Q_U\shom^\bullet_U(u^*\CG^\bullet, \CK^\bullet)) \\
    & \cong \R{}u_* \R{}Q_U\R\shom^\bullet_U(u^* \CG^\bullet, \CK^\bullet) 
               \tag{by Lemma \ref{agaf} for $u$}\\
    & \cong \R{}Q_X\R{}u_*\R\shom^\bullet_U(u^*\CG^\bullet, \CK^\bullet) \\
    & \cong \R{}Q_X\R\shom^\bullet_X(\CG^\bullet, \R{}u_* \CK^\bullet) \\
    & \cong \R{}Q_X\R\shom^\bullet_X(\CG^\bullet, u_* \CK^\bullet)
               \tag{by the previous Lemma}.
\end{align*}

Take now $U$ such that $s(U) = n > 1$. Take a finite covering $U = \bigcup_{i=1}^n U_i$ where every $U_i$ is an affine open subset of $X$. Denote by $u_i \colon U_i \inc X$ and by $u'_i \colon U_i \inc U$ the canonical embeddings. Let $V := \bigcup_{i=2}^{n} U_i$ and $W := U_1 \cap V$. It is clear that $s(V) = n-1$. Also, observe that $W = \bigcup_{i=2}^{n} (U_i \cap U_1)$ and the open subsets $U_i \cap U_1$ are affine for every $i \in \{2, \dots, n\}$ by semi-separation (Proposition \ref{sskey}), therefore $s(W) \leq n-1$. Let $v \colon V \inc X$, $w \colon W \inc X$, $v' \colon V \inc U$ and $w' \colon W \inc U$ denote the canonical open embeddings. As $V$ and $W$ are quasi-compact the maps $v$, $w$ and $u_1$ are quasi-compact and quasi-separated. Consider the distiguished triangle obtained applying the functor $\R{}u_{*}$ to the the Mayer-Vietoris triangle in $\D(U) $ associated to $\CK^\bullet \in \D(\A_\qc(U))$ for the open cover $U = V \cup W$.
\[
\CK^\bullet \lto 
\R{}u'_{1 \, *}{u'_1}^*\CK^\bullet \oplus \R{}v'_{*}{v'}^*\CK^\bullet \lto 
\R{}w'_{*}{w'}^*\CK^\bullet \overset{+}{\lto}
\]
(\cfr \cite[Proof of Lemma (4.7.5.1), p. 188]{LDC} and \cite[Proof of Corollary (1.3.1)]{AJL1}). By Lemma \ref{agfst}, it yields a triangle in $\D(\A_\qc(X))$ (ommiting $\mathbf{i}_X, \mathbf{i}_U \dots$ from the notation)
\[
\R_\qc{}u_{*}\CK^\bullet \lto 
\R_\qc{}u_{1 \, *}{u'_1}^*\CK^\bullet \oplus \R_\qc{}v_{*}{v'}^*\CK^\bullet \lto 
\R_\qc{}w_{*}{w'}^*\CK^\bullet \overset{+}{\lto}
\]
Let $\CK_1^\bullet := \R_\qc{}u_{1 \, *}{u'_1}^*\CK^\bullet$, $\CK_2^\bullet := \R_\qc{}v_{*}{v'}^*\CK^\bullet$ and $\CK_{1\,2}^\bullet := \R_\qc{}w_{*}{w'}^*\CK^\bullet$, then the previous triangle becomes
\[
\R_\qc{}u_{*}\CK^\bullet \lto 
\CK_1^\bullet \oplus \CK_2^\bullet \lto 
\CK_{1\,2}^\bullet \overset{+}{\lto}.
\]
Applying the natural transformation $\phi_{X,\CG^\bullet}$ we obtain the following commutative diagram (that we display transposed)
\begin{diagram}[height=1.8em,w=4em,labelstyle=\scriptstyle]
\R_\qc(Q_X \shom^\bullet_X(\CG^\bullet, \R_\qc{}u_{*}\CK^\bullet)) & 
\rTo^{\phi} &
\dhom^\bullet_X(\CG^\bullet, \R_\qc{}u_{*}\CK^\bullet)\\
\dTo & &\dTo \\
\R_\qc(Q_X  \shom^\bullet_X(\CG^\bullet, \CK_1^\bullet \oplus \CK_2^\bullet)) & \rTo^{\phi_1 \oplus \phi_2} &
\dhom^\bullet_X(\CG^\bullet, \CK_1^\bullet \oplus \CK_2^\bullet)\\
\dTo & &\dTo \\
\R_\qc(Q_X \shom^\bullet_X(\CG^\bullet, \CK_{1\,2}^\bullet))  & 
\rTo^{\phi_{1\,2}} &
\dhom^\bullet_X(\CG^\bullet, \CK_{1\,2}^\bullet)\\
\dTo_+ & &\dTo_+ \\
&&
\end{diagram}
Note that $\phi_2$ and $\phi_{1\,2}$ are isomorphisms in $\D(\A_\qc(X))$ by the induction hypothesis, while $\phi_1$ is an isomorphism because $s(U_1) =1$. We conclude that $\phi$ is an isomorphism as wanted.
\end{proof}

\section{The case of formal schemes}

In this last section, we will prove how to transport the previous results to the context of formal schemes and quasi-coherent torsion sheaves as defined in \cite{dfs}. Let us recall briefly some basic facts and notations that will be used throughout this section.

Let $(\FX, \CO_\FX)$ be a noetherian formal scheme with an ideal of definition  $\CI$. Denote by $\CA(\FX)$ the category of all $\CO_\FX$-modules. Consider the functor $\varGamma'_\FX \colon \CA(\FX) \to \CA(\FX)$ defined by
\[ 
\varGamma'_\FX \CF := \dirlim{n > 0} \shom_{\CO_\FX}(\CO_\FX/\CI^n, \CF)
\] 
for $\CF \in \CA(\FX)$. This functor does not depend on $\CI$ but only on the topology of the sheaf of rings $\CO_{\FX}$.  Let $\CA_\ts(\FX)$ be the full subcategory of $\CA(\FX)$ consisting of sheaves $\CF$ such that
$\varGamma'_\FX \CF = \CF$; it is a \emph{plump} subcategory of $\CA(\FX)$.
This means it is closed for kernels, cokernels and extensions (\emph{cfr.}
\cite[beginning of \S 1]{dfs}). We will consider the subcategory 
$\CA_\qct(\FX) : = \CA_\ts(\FX) \cap \CA_\qc(\FX)$. It is again a plump
subcategory of $\CA(\FX)$ by \cite[Corollary 5.1.3]{dfs} and it defines a
triangulated subcategory of the derived category $\D(\FX) := \D(\CA(\FX))$, namely $\D_\qct(\FX)$, the full subcategory of
$\D(\FX)$ formed by complexes whose homology lies in $\CA_\qct(\FX)$.
The inclusion functor $\CA_\qct(\FX) \to \CA(\FX)$ has a right adjoint
denoted $Q^\ts_\FX$ (see \cite[Corollary 5.1.5]{dfs}). By taking
of K-injective resolutions we have a $\Delta$-functor 
$\R Q^\ts_\FX : \D(\FX) \to \D(\CA_\qct(\FX))$.

\begin{lem}
Let $\FX$ be a locally noetherian formal scheme, the category $\A_\qct(\FX)$ is a Grothendieck category.
\end{lem}

\begin{proof}
The fact that it is abelian and possesses direct limits follows from \cite[Corollary 5.1.3]{dfs}. Indeed, the category $\A_\qct(\FX)$ is a plump subcategory of $\A(\FX)$ and a plump subcategory of an abelian category is abelian, being stable for kernels, cokernels and extensions (\cfr \cite[(1.9.1)]{LDC}). Direct limits in $\A_\qct(\FX)$ are exact because they are just direct limits in $\A(\FX)$. Finally, the existence of a generating set follows from \cite[Lemma 5.1.4]{dfs}, a set of generators is given by a set of representatives of isomophism classes of \emph{coherent} torsion sheaves.
\end{proof}

\begin{thm} \label{fs-asht145}
Let $\FX$ be a locally noetherian formal scheme. For the category $\D(\A_\qct(\FX))$ the properties (\ref{tri}),  (\ref{cop}) and (\ref{brw}) always hold.
\end{thm}

\begin{proof}
This is similar to the proof of Theorem \ref{asht145}.
Once again, (\emph{\ref{tri}}) is trivial because a derived category is triangulated.
The existence of coproducts (\emph{\ref{cop}}) is due to the fact that a coproduct in the category $\A(\FX)$ of sheaves belonging to $\A_\qct(\FX)$, remains quasi-coherent torsion by \cite[Corollary (5.1.3)]{dfs}. But coproducts of sheaves are exact as recalled before.
Finally, (\emph{\ref{brw}}) is satisfied using again \cite[Theorem 5.8]{AJS}.
\end{proof}

\begin{cosa}
Let $f\colon\FX \to \FY$ be a map of locally noetherian formal schemes and let $\CJ \subset \CO_{\FX}$ and $\CK \subset \CO_{\FY}$ be ideals of definition such that $f^{*}(\CK)\CO_{\FX} \subset \CJ$. If $f_{n}\colon X_{n} \to Y_{n}$ with $X_{n}:=(\FX,\CO_{\FX}/\CJ^{n+1})$ and $Y_{n}:=(\FY,\CO_{\FY}/\CK^{n+1})$ is the morphism induced by $f$, for each $n \in \NN$, then $f$ can be expressed as \cite[\S10.6]{ega1}\footnote{\cite[\S 10.6]{ega160}} 
\[f = \dirlim {n\in \NN} f_{n}.\]

We say that a locally noetherian formal scheme $\FX$ is \emph{semi-separated} if the diagonal map $\Delta_f \colon \FX \to \FX \times_{\spec{\ZZ}} \FX$ is an affine morphism where $f \colon \FX \to \spec{\ZZ}$ is the canonical map.
By \cite[Proposition (10.16.2)]{ega1}\footnote{Unfortunately, there is no reference for this in \cite{ega160}} the morphism $f$ is affine if, and only if, $f_0$ is. As a consequence, we can transport the results of \S2 to the context of locally noetherian formal schemes. In particular, the following is true:
\end{cosa}

\begin{prop}\label{fs-sskey}
Let $\FX$ be a locally noetherian formal scheme and $\{\FU_{\alpha}\}_{\alpha \in L}$ be an affine open covering of $\FX$. The formal scheme $\FX$ is semi-separated if, and only if, for any pair of indices $\alpha, \beta \in L$ the open subset $\FU_{\alpha} \cap \FU_{\beta}$ is affine.
\end{prop}

\begin{proof}
Follows immediately from the previous remark and Proposition \ref{sskey}. \noqed
\end{proof}

\begin{cosa} \label{fs-equiv}
We will now recall briefly for noetherian formal schemes the notions given in Section \ref{closed} for ordinary schemes. For every $\CF^\bullet \in \D(\FX)$ there exists a K-flat resolution $\CP_\CF^\bullet \iso \CF^\bullet$ \cite[Proposition 5.6]{S}. As a consequence there exist a derived functor
\[
\CF \otimes^{\LL}_{\CO_\FX}\!\!\!- \colon 
\D(\FX) \times \D(\FX) \to \D(\FX)
\]
defined by 
$\CF^\bullet \otimes^{\LL}_{\CO_\FX} \CG^\bullet
\cong \CP_\CF^\bullet \otimes^{\LL}_{\CO_\FX} \CG^\bullet$. Given $\CF^\bullet, \CG^\bullet \in \D_\qct(\FX)$ the complex $\CF^\bullet \otimes^{\LL}_{\CO_\FX} \CG^\bullet$ has quasi-coherent torsion homologies. Indeed, by \cite[Proposition 5.2.1 (a)]{dfs}, this is a local question and we can assume that $\FX = \spf(A)$, where $A$ is a noetherian $I$-adic ring. Let $X = \spec(A)$ and $\kappa \colon \FX \to X$ be the completion morphism. Let $Z = V(I)$ be the closed subscheme of $X$ associated to the ideal of definition. The exact functors $\kappa^*$ and $\kappa_*$ restrict to inverse isomorphisms between the categories  $\D_\qct(\FX)$ and $\D_{\qc Z}(X) \subset \D_\qc(X)$ of those complexes whose homologies are supported in $Z$ \cite[Proposition 5.2.4]{dfs}. Thus for $\CF^\bullet, \CG^\bullet \in \D_\qct(\FX)$ we have
\[
\CF^\bullet \otimes^{\LL}_{\CO_\FX} \CG^\bullet \cong
\kappa^*\kappa_*\CF^\bullet \otimes^{\LL}_{\CO_\FX} \kappa^*\kappa_*\CG^\bullet \cong
\kappa^*(\kappa_*\CF^\bullet \otimes^{\LL}_{\CO_X} \kappa_*\CG^\bullet)
\]
and the result follows from the fact that 
$\kappa_*\CF^\bullet \otimes^{\LL}_{\CO_\FX} \kappa_*\CG^\bullet \in \D_{\qc Z}(X)$
by  \cite[Corollary 3.1.2]{AJL1} and \cite[(2.5.8)]{LDC}. This allows us to define a bifunctor
\[
-\otimes^{\LL}_{\CO_X}\!\!\!- \colon 
\D_\qct(\FX) \times \D_\qct(\FX) \to \D_\qct(\FX)
\]
together with the usual associativity coherence inherited from the bifunctor $-\otimes^{\LL}_{\CO_X}\!\!\!-$ in $\D(\FX)$.

Moreover, if $\FX$ is semi-separated (or of finite Krull dimension) the functor  $\R{}Q_\FX^\ts$ provides an equivalence of categories  between $\D_\qct(\FX)$ and $\D(\CA_\qct(\FX))$. To see it, it is enough to realize that in the proof of \cite[Proposition 5.3.1]{dfs}, separated case, it is only used the fact that the intersection of two open affine subsets is again affine (the semi-separation of $\FX$).
\end{cosa}

\begin{cosa}
The category $\D_\qc(\FX)$ is not well-behaved unless $\FX$ is an ordinary scheme. It is not guaranteed that  $\CA_\qc(\FX)$ ---and therefore $\D_\qc(\FX)$--- has all coproducts, see \cite[\S 3]{dfs}. On the contrary the subcategory 
$\D_\qct(\FX) \subset \D_\qc(\FX)$ has good properties, see \lc \S 3 and \S 5, and constitutes a suitable choice of cohomological coefficients for formal schemes. 

The only drawback of this choice is that $\D_\qct(\FX)$ does not contain the category $\D_\cc(\FX)$ ---the derived category of complexes with coherent cohomology. This can be repaired as follows. Denote by $\BG := \R\varGamma'_\FX$ the right-derived functor of $\varGamma'_\FX$ and by $\BL$ the functor
\[
\BL := \R\shom^\bullet_{\CO_\FX}(\R\varGamma'_\FX \CO_\FX,-) \colon
\D_\qc(\FX) \to \D(\FX).
\]
Consider the category $\widehat{\D}(\FX)$ defined as the essential image of the functor $\BL$ and $\widehat{\D}_\qc(\FX)$ the subcategory of $\widehat{\D}(\FX)$ which corresponds to the essential image of the functor $\BL$ restricted to $\D_\qc(\FX)$. The functors $\BL$ and $\BG$ restrict to quasi-inverse equivalences 
between the categories $\D_\qct(\FX)$ and $\widehat{\D}_\qc(\FX)$ \cite[Remarks 6.3.1 (1), (3)]{dfs}.

By the previous discussion, the results that we proved on $\D_\qct(\FX)$ could be transported by the equivalence to $\widehat{\D}_\qc(\FX)$. From now on, we will concentrate on $\D_\qct(\FX)$. We warn the reader that these equivalences are not compatible with their canonical inclusions in $\D(\FX)$. Observe that $\D_\cc(\FX)$ is now a subcategory of $\widehat{\D}_\qc(\FX)$ because for all $\CF^\bullet \in \D_\cc(\FX)$ the canonical map $\CF^\bullet \to \BL\CF^\bullet$ is an isomorphism as follows from applying \cite[Proposition 6.2.1]{dfs} to the case $\CE = \CO_\FX$. 
\end{cosa}

\begin{rem}
Let $\A_{\vec{\cc}}(\FX)$ be the category formed by those sheaves in $\A(\FX)$  which are direct limits of its coherent submodules. The categories $\A_\qct(\FX) \subset \A_{\vec{\cc}}(\FX) \subset \A_\qc(\FX)$ are plump subcategories of $\A(\FX)$. The category $\D_{\vec{\cc}}(\FX) \subset \D_\qc(\FX)$ is the subcategory formed by the subcomplexes whose homologies belong to $\A_{\vec{\cc}}(\FX)$, see \lc beginning of \S 1. We have that
\[\BG(\D_{\vec{\cc}}(\FX)) = \BG(\D_\qc(\FX)) = \D_\qct(\FX),\]
so the relation $\BL = \BL \BG$ in \lc Remark 6.3.1 (1) shows that 
\[\BL(\D_\qct(\FX)) = \BL(\D_{\vec{\cc}}(\FX)) = \BL(\D_\qc(\FX)) = \widehat{\D}_\qc(\FX).\] In \cite[Theorem 0.1]{fgm} it is proved that $\BL|_{\D_{\vec{\cc}}(\FX)}$ is isomorphic to the left derived functor of the completion functor. That is why we interpret $\widehat{\D}_\qc(\FX)$ as a completion of $\D_\qc(\FX)$ for the canonical topology of $\CO_\FX$. Note that $\BL\BL = \BL$ \cite[(b) in Remark 6.3.1 (1)]{dfs}.
\end{rem}

\begin{cosa}
Let $\FX$ be a semi-separated noetherian formal scheme. In $\D_\qct(\FX)$ there is an internal hom defined (and to our knowledge introduced) in \cite[\S1.2]{AJSB} as:
\[\dhom^\bullet_{\FX}(\CF^\bullet,\CG^\bullet) :=
\R{}Q^\ts_\FX\R\shom^\bullet_{\CO_\FX}(\CF^\bullet,\CG^\bullet).
\]  
for $\CF^\bullet, \CG^\bullet \in \D_\qct(\FX)$. Using the same techniques as in Proposition \ref{inthom37} the reader can check that 
\[
 \Hom_{\D_\qct(\FX)}(\CF^\bullet, \dhom^\bullet_\FX(\CG^\bullet, 
 \CH^\bullet)) \iso
 \Hom_{\D_\qct(\FX)}(\CF^\bullet \otimes^{\LL}_{\CO_\FX} \CG^\bullet, 
 \CH^\bullet).
\]
\ie\!\!\!, for $\CG \in \D_\qct(\FX)$ we have an adjunction $-\otimes^{\LL}_{\CO_\FX}\CG \dashv \dhom^\bullet_\FX(\CG, -)$ in $\D_\qct(\FX)$.
\end{cosa}

\begin{thm} \label{fs-asht2}
Let $\FX$ be a noetherian semi-separated formal scheme. The category $\D_\qct(\FX)$ has a natural structure of symmetric closed category. In other words, property (\ref{scl}) of \ref{axioms} holds.
\end{thm}

\begin{proof}
The unit object for $-\otimes^{\LL}_{\CO_\FX}\!\!\!-$ is $\R\varGamma'_\FX \CO_\FX$ 
. We remark that there is a canonical isomorphism $\R{}Q^\ts_\FX\R\varGamma'_\FX \CO_\FX \cong \R\varGamma'_\FX \CO_\FX$ in $\D_\qct(\FX)$. We will denote this object by $\CO'_\FX$ for convenience. Now reasoning as in Theorem \ref{asht2}, the data $(\D_\qct(\FX), \otimes^{\LL}_{\CO_\FX}, \CO'_X)$ together with the corresponding compatibility diagrams define a monoidal category because using \cite[Examples (3.5.2) (d)]{LDC} this is the case for $(\D(\FX), \otimes^{\LL}_{\CO_\FX}, \CO_\FX)$. The category $\D_\qct(\FX)$ is a full subcategory of $\D(\FX)$, and it is stable for the tensor product. To see that $\CO'_\FX$ is the unit object, consider the following chain of isomorphisms
\begin{align*}
\CF \cong \CF\otimes^{\LL}_{\CO_\FX}\CO_\FX & 
      \cong \R\varGamma'_\FX\CF\otimes^{\LL}_{\CO_\FX}\CO_\FX 
      \tag{$\CF \in \D_\qct(\FX)$}\\
    & \cong\CF\otimes^{\LL}_{\CO_\FX}\R\varGamma'_\FX\CO_\FX 
    \tag{\cite[Corllary (3.1.2)]{AJL1}}\\
    & = \CF\otimes^{\LL}_{\CO_\FX}\CO'_\FX.
\end{align*}
And analogously $\CF \cong \CO'_\FX\otimes^{\LL}_{\CO_\FX}\CF$. The adjunction $-\otimes^{\LL}_{\CO_\FX}\CG \dashv \dhom^\bullet_\FX(\CG, -)$, holds for any $\CG \in \D_\qct(\FX)$ by the discussion in the previous paragraph. Again, it is clear that both bifunctors are $\Delta$-functors in either variable and that the square
\begin{diagram}[height=2em,w=2em,p=0.3em,labelstyle=\scriptstyle]
\CO'_\FX[r] \otimes^{\LL}_{\CO_\FX} \CO'_\FX[s] & \rTo^{\sim}  
                                                 &  \CO'_\FX[r + s] \\
\dTo_\theta                             &              & \dTo_{(-1)^{rs}} \\ 
\CO'_\FX[s] \otimes^{\LL}_{\CO_\FX} \CO'_\FX[r] & \rTo^{\sim}  
                                        &  \CO'_\FX[r + s] \\
\end{diagram}
commutes by \cite[(1.5.4.1)]{LDC}.
\end{proof}

\begin{cosa}
Let $\FX$ be a formal scheme. A complex $\CE^\bullet \in \CCC(\FX)$ is called \emph{perfect} if for every $x \in \FX$ there is an open neighborhood $\FU$ of $x$ and a bounded complex of locally-free finite type modules $\CF^\bullet$ together with a quasi-isomorphism $\CF^\bullet \to \CE^\bullet|_\FU$ in $\CCC(\FU)$ or, what amounts to the same, an isomorphism $\CF^\bullet \iso \CE^\bullet|_\FU$ in $\D(\FU)$ (\cfr \cite[Corollaire 4.3]{I1}).
\end{cosa}

\begin{prop} \label{seisocho}
Let $\FX = \spf(A)$ be an affine formal scheme such that $A$ is an $I$-adic noetherian ring. The category $\D_\qct(\FX)$ is generated by a compact object.
\end{prop}

\begin{proof}
Using \cite[Propositions 5.2.4]{dfs} it follows easily that the category $\D_\qct(\FX)$ is equivalent to $\D_{I}(A)$, the full subcategory of the derived category of $A$-modules such that its homologies are $I$-torsion. But this triangulated category is generated by any Koszul complex $K^\bullet$ associated to a sequence of generators of $I$ by \cite[Proposition 6.1]{BN}, or, for a somehow more detailed proof, \cite[Proposition 6.1]{dg}. Now $K^\bullet$ is a perfect complex of $A$-modules (in fact, it is strictly perfect), by \cite[Lemma 4.3]{AJST} it is also a compact object in $\D(A)$, and therefore in $\D_{I}(A)$. We conclude that $\D_{I}(A)$ is generated by a compact object, or, what amounts to the same so does the category $\D_\qct(\FX)$.
\end{proof}

\begin{lem}\label{fs-mj} 
Let $\FX$ be a noetherian formal scheme and let $\FU$ be an affine open subset of $\FX$ and denote by $j \colon \FU \inc \FX$, the canonical inclusion. If $\CE^\bullet$ is a compact object in $\D_\qct(\FX)$ then its restriction, $j^*\CE^\bullet$, is a compact object in $\D_\qct(\FU)$.
\end{lem}

\begin{proof}
By \cite[Proposition 5.2.6]{dfs} we have that $\R{}j_*$ takes complexes in $\D_\qct(\FU)$ into complexes in $\D_\qct(\FX)$.
Let $\{\CF^{\bullet}_\alpha \, / \, \alpha \in A\}$ be a set of objects in $\D_\qct(\FU)$. By \cite[Proposition 3.5.2]{dfs} we have an isomorphism 
\[\phi \colon \oplus_{\alpha \in A} \R{}j_*\CF^{\bullet}_\alpha \,\iso\, \R{}j_*\oplus_{\alpha \in A}\CF^{\bullet}_\alpha,\] therefore
\begin{align*}
\oplus_{\alpha \in A}\Hom_{\D(\FU)}&(j^*\CE^\bullet, 
\CF^{\bullet}_\alpha)
\\
    & \cong \oplus_{\alpha \in A}\Hom_{\D(\FX)}(\CE^{\bullet}, 
            \R{}j_*\CF^{\bullet}_\alpha)
                           \tag{$j^* \dashv  \R{}j_*$}\\ 
    & \cong \Hom_{\D(\FX)}(\CE^{\bullet}, 
            \oplus_{\alpha \in A}\R{}j_*\CF^{\bullet}_\alpha) 
                          \tag{$\CE^{\bullet}$ compact}\\
    & \cong \Hom_{\D(\FX)}(\CE^{\bullet}, 
            \R{}j_*\oplus_{\alpha \in A}\CF^{\bullet}_\alpha) 
                          \tag{$\phi$ isomorphism}\\
    & \cong \Hom_{\D(\FU)}(j^*\CE^\bullet, 
            \oplus_{\alpha \in A}\CF^{\bullet}_\alpha). 
                                    \tag{$j^* \dashv  \R{}j_*$}
\end{align*}
Note that the map $j$ is adic because it is an open embedding. By \cite[Corollary 5.2.11 (b)]{dfs} the adjunction $j^* \dashv  \R{}j_*$ restricts to the subcategory of objects with quasi-coherent torsion homologies. So we conclude that $j^*\CE^\bullet$ is a compact object.
\end{proof}

\begin{prop} \label{fs-compsd}
Let $\FX$ be a noetherian formal scheme. The compact objects in $\D_\qct(\FX)$ are the perfect complexes.
\end{prop}

\begin{proof}
Let us see that perfect implies compact. Let $\CE^\bullet$ be a perfect complex and let $\{\CF^\bullet_\lambda \,/\, \lambda \in \Lambda\}$ be a family of complexes in $\D_\qct(\FX)$. We will see first that the canonical map 
\[
\phi \colon
\oplus_{\lambda \in \Lambda} \rshom_\FX(\CE^\bullet, \CF^\bullet_\lambda) \lto
\rshom_\FX(\CE^\bullet, \oplus_{\lambda \in \Lambda} \CF^\bullet_\lambda)
\]
is an isomorphism in $\D(\FX)$. This is a local question therefore we may take a point $x \in \FX$ and an open neighborhood $\FV \subset \FX$ of $x$ such that $\CE^\bullet |_\FV$ is a bounded complex of free finite rank modules. Take $\FV$ for $\FX$ and let us check that $\phi$ is an isomorphism. But this is clear. If the complex $\CE^\bullet$ has length one then, it is trivial. If the complex has length $n > 1$, suppose that $q \in \ZZ$ is the first integer such that $\CE^q \neq 0$, that exists because $\CE^\bullet$ is bounded. Then there is a distinguished triangle $\CE^q[-q] \to \CE^\bullet \to {\CE'}^\bullet \overset{+}\to$ with ${\CE'}^\bullet$ of length $n - 1$. The fact holds for $\CE^q[-q]$ and  for ${\CE'}^\bullet$ by induction, therefore it has to hold for $\CE^\bullet$. Arguing as in \cite[Theorem 2.4.1]{tt} we have that 
$\rshom_\FX(\CE^\bullet, \CF^\bullet_\lambda) \in \D_\qct(\FX)$. Now we have the following chain of canonical isomorphisms
\begin{align*}
\oplus_{\lambda \in \Lambda} &\Hom_{\D(\FX)}(\CE^\bullet, \CF^\bullet_\lambda) \cong \\
    & \cong \h^0(\oplus_{\lambda \in \Lambda} \rhom^\bullet_\FX(\CE^\bullet, \CF^\bullet_\lambda)) 
                   \tag{$\h^0$ commutes with $\oplus$}\\
    & \cong \h^0(\oplus_{\lambda \in \Lambda} \R\Gamma(\FX, \rshom^\bullet_\FX(\CE^\bullet, \CF^\bullet_\lambda))) \\
    & \cong \h^0(\R\Gamma(\FX, \oplus_{\lambda \in \Lambda}\rshom^\bullet_\FX(\CE^\bullet, \CF^\bullet_\lambda))) 
                   \tag{by \cite[Proposition 3.5.2]{dfs}}\\
    & \cong \h^0(\R\Gamma(\FX, \rshom^\bullet_\FX(\CE^\bullet, \oplus_{\lambda \in \Lambda}\CF^\bullet_\lambda)))
                   \tag{via $\phi$}\\
    & \cong \Hom_{\D(\FX)}(\CE^\bullet, \oplus_{\lambda \in \Lambda} \CF^\bullet_\lambda). 
\end{align*}
which show that $\CE^\bullet$ is compact in $\D_\qct(\FX)$.

Conversely, let us see that a compact object $\CE^\bullet\in \D_\qct(\FX)$ is a perfect complex. Let us assume first that $\FX$ is affine, \ie $\FX = \spf(A)$ where $A$ is an $I$-adic noetherian ring. Consider the completion morphism $\kappa \colon \spf(A) \to \spec(A)$. Let $K^\bullet$ be a Koszul complex associated to a sequence of generators of $I$. By the proof of Proposition \ref{seisocho} $\CK^\bullet := \kappa^*\widetilde{K^\bullet}$ generates $\D_\qct(\FX)$, and by \ref{strong} the smallest triangulated subcategory stable for coproducts containing $\CK^\bullet$ is all of $\D_\qct(\FX)$. Applying \cite[Lemma 2.2]{Ntty} the thick subcategory\footnote{Thick subcategory = triangulated and stable for direct summands.} of $\D_\qct(\FX)$ formed by its compact objects is the smallest one that contains $\CK^\bullet$ but all its objects are perfect because the subcategory of perfect complexes of $\D_\qct(\FX)$ is thick and contains $\CK^\bullet$ as can be seen adapting the argument in \cite[Proposition 2.2.13]{tt}.

Assume now that $\FX$ is a noetherian formal scheme and that $\CE^\bullet$ is a compact object in $\D_\qct(\FX)$. Let $\FU$ be an affine open subset of $\FX$ and denote by $j \colon \FU \inc \FX$ the canonical inclusion. By Lemma \ref{fs-mj} $j^*\CE^\bullet = \CE^\bullet |_\FU$ is compact and by the previous discussion it is perfect, but being perfect is a local question, therefore $\CE^\bullet$ is a perfect complex.
\end{proof}

\begin{lem}
Let $\FX$ be a noetherian formal scheme and let $\FU$ be a open subset of $\FX$. Let $\CF^\bullet \in \D_\qct(\FX)$ and $\CG^\bullet \in \D_\qct(\FU)$ a compact object. If $\alpha \colon \CG^\bullet \to \CF^\bullet |_\FU$ is a morphism in $\D_\qct(\FU)$ there exist compact objects ${\CG'}^\bullet \in \D_\qct(\FU)$ and 
$\widetilde{\CG}^\bullet \in \D_\qct(\FX)$, an isomorphism $\beta \colon   \widetilde{\CG}^\bullet |_\FU \iso \CG^\bullet \oplus {\CG'}^\bullet$ and a morphism $\widetilde{\alpha} \colon \widetilde{\CG}^\bullet \to \CF^\bullet$ in $\D_\qct(\FX)$ such that it extends $\alpha \circ \pi_1 \colon \CG^\bullet \oplus {\CG'}^\bullet \to \CF^\bullet |_\FU$, \ie there is an isomorphism $\beta \colon \widetilde{\CG}^\bullet |_\FU \to \CG^\bullet \oplus {\CG'}^\bullet$ such that the diagram
\begin{diagram}[height=2em,w=2.8em,p=0.3em,labelstyle=\scriptstyle]
\widetilde{\CG}^\bullet |_\FU     &                           &   \\
\dTo^{\beta}_{\wr}                & \rdTo^{\widetilde{\alpha} |_\FU}  & \\ 
\CG^\bullet \oplus {\CG'}^\bullet & \rTo^{\alpha \circ \pi_1} & \CF^\bullet |_\FU \\
\end{diagram}
commutes.
\end{lem}

\begin{proof}
Suppose first that $\FX$ is affine, \ie $\FX = \spf(A)$ where $A$ is an $I$-adic noetherian ring. Let $\kappa \colon \FX \to \spec(A)$ be the completion morphism. Put $\FZ := \FX \setminus \FU$ and $\ia \supset I$ be an open ideal of $A$ such that $\FZ =V(\ia)$. Let $\CJ := \kappa^*\widetilde{\ia}$ and let $\D_\CJ(\FX)$ be the full subcategory of $\D(\FX)$ whose objects are the complexes with $\CJ$-torsion homology \cite[Proposition 5.2.8]{dfs}. Set $\D_{\qct \CJ}(\FX) = \D_\CJ (\FX) \cap \D_{\qct}(\FX)$. The equivalence of categories $\kappa_* \colon \D_\qct(\FX) \to \D_I(A)$ restricts to an equivalence between $\D_{\qct \CJ}(\FX)$ and $\D_{\ia}(A)$. Thus the triangulated category $\D_{\qct \CJ}(\FX)$ is generated by a compact object. Applying \cite[Theorem 2.1]{Ngd} we obtain a compact object $\widetilde{\CG}^\bullet \in \D_\qct(\FX)$ and a morphism $\widetilde{\alpha} \colon \widetilde{\CG}^\bullet \to \CF^\bullet$ in $\D_\qct(\FX)$ that extends 
$\alpha \circ \pi_1 \colon \CG^\bullet \oplus {\CG}^\bullet[1] \to \CF^\bullet |_\FU$.

If $\FX$ is not affine, then let $\FX = \FU \cup \FW_1 \cup \dots \cup \FW_n$, where $n \geq 1$ and $\FW_1 \dots \FW_n$ are affine open subsets of $\FX$. By the affine case we know that there is a compact object $\CG^\bullet_1 \in \D_\qct(\FW_1)$ and a morphism $\alpha' \colon \CG^\bullet_1 \to \CF^\bullet|_{\FW_1}$ in $\D_\qct(\FW_1)$ that extends $\alpha \circ \pi_1|_{\FU \cap \FW_1}$ to $\FW_1$. Let $u \colon \FU \to \FU \cup \FW_1$, $w \colon \FW_1 \to \FU \cup \FW_1$ and
$j \colon \FU \cap \FW_1 \to \FU \cup \FW_1$ be the canonical inclusion maps. In this setting, we obtain a morphism of triangles
\begin{diagram}[height=2em,w=4em,labelstyle=\scriptstyle]
\widetilde{\CG}^\bullet_1&\rTo&\R{}u_*(\CG^\bullet \oplus {\CG}^\bullet[1])\oplus\R{}w_*\CG^\bullet_1&\rTo&\R{}j_*(j^*(\CG^\bullet \oplus {\CG}^\bullet[1]))&\rTo^+\\
\dTo^{\widetilde{\alpha}_1}    &&\dTo^{\text{ via } \alpha \circ \pi_1 \text{ and } \alpha' \text{ }}&&\dTo^{\text{via } \alpha \circ \pi_1}&\\
\CF^\bullet|_{\FU \cup \FW_1}&\rTo&\R{}u_*(\CF^\bullet|_\FU)\oplus\R{}w_*(\CF^\bullet|_{\FW_1})&\rTo&\R{}j_*(j^*\CF^\bullet)&\rTo^+
\end{diagram}
Since $\widetilde{\CG}^\bullet_1|_\FU \cong \CG^\bullet \oplus {\CG}^\bullet[1]$ and $\widetilde{\CG}^\bullet_1|_{\FW_1} \cong  \CG^\bullet_1$ then $\widetilde{\CG}^\bullet_1$ is a compact object in $\FU \cup \FW_1$. Furthermore $\widetilde{\alpha}_1$ extends $\alpha \circ \pi_1 \colon \CG^\bullet \oplus {\CG}^\bullet[1] \to \CF^\bullet |_\FU$ to $\FU \cup \FW_1$.

If $n = 1$ take $\widetilde{\alpha}_1 = \widetilde{\alpha}$. If $n > 1$, proceed as before in $n$ steps obtaining a compact object $\widetilde{\CG}^\bullet$ together with a morphism $\widetilde{\alpha} \colon \widetilde{\CG}^\bullet \to \CF^\bullet$ in $\D_\qct(\FX)$ satisfying the desired conditions.
\end{proof}

\begin{prop} \label{fs-cgenss}
Let $\FX$ be a noetherian formal scheme. The category $\D_\qct(\FX)$ is generated by compact objects.
\end{prop}

\begin{proof}
Let $\CF^\bullet \in \D_\qct(\FX)$, $\CF^\bullet \neq 0$. There exists an affine open subset $\FU \subset \FX$ such that $\CF^\bullet |_\FU \neq 0$. By Proposition \ref{seisocho}, $\D_\qct(\FU)$ is compactly generated, therefore there is a compact object $\CG^\bullet \in \D_\qct(\FU)$ together with a map $\alpha \colon \CG^\bullet \to \CF^\bullet |_\FU$ such that $\alpha \neq 0$. By the previous lemma $\alpha$ provides a non zero morphism $\widetilde{\alpha} \colon \widetilde{\CG}^\bullet \to \CF^\bullet$ in $\D_\qct(\FX)$ with $\widetilde{\CG}^\bullet$ a compact object.
\end{proof}

\begin{lem}
Let $\FX$ be a noetherian semi-separated formal scheme, $\CE^\bullet\in \D_\qct(\FX)$ and $\CF^\bullet \in \D(\FX)$. We have the following isomorphism:
\[
\dhom^\bullet_\FX(\CE^\bullet, \R\varGamma'_\FX\CF^\bullet) \liso
\dhom^\bullet_\FX(\CE^\bullet, \CF^\bullet).
\]
\end{lem}

\begin{proof}
It is enough to check that there is an isomorphism
\[
\R\shom^\bullet_\FX(\CE^\bullet, \R\varGamma'_\FX\CF^\bullet) \liso
\R\shom^\bullet_\FX(\CE^\bullet, \CF^\bullet)
\]
and this follows at once from \cite[Corollary 5.2.3]{dfs}.
\end{proof}

\begin{prop} \label{fs-perfsd}
Let $\FX$ be a noetherian semi-separated formal scheme and $\CE^\bullet \in  \D_\qct(\FX)$ a perfect complex, then $\CE^\bullet$ is strongly dualizable.
\end{prop}

\begin{proof}
Let $\CG^\bullet \in \D_\qct(\FX)$. We need to check that
\[
\dhom^\bullet_\FX(\CE^\bullet, \CO'_\FX) \otimes^{\LL}_{\CO_\FX} \CG^\bullet \lto
\dhom^\bullet_\FX(\CE^\bullet, \CG^\bullet)
\]
is an isomorphism. We have remarked in the proof of Theorem \ref{fs-asht2} that $\CO'_\FX = \R\varGamma'_\FX \CO_\FX$. The fact that $\CE^\bullet$ is perfect implies that $\rshom^\bullet_\FX(\CE^\bullet, \R\varGamma'_\FX \CO_\FX)$ and $\rshom^\bullet_\FX(\CE^\bullet, \CG^\bullet)$ belong to
$\D_\qct(\FX)$. Therefore we are reduced to prove that the canonical map
\[
\rshom^\bullet_\FX(\CE^\bullet, \R\varGamma'_\FX \CO_\FX) 
\otimes^{\LL}_{\CO_\FX} \CG^\bullet \lto
\rshom^\bullet_\FX(\CE^\bullet, \CG^\bullet)
\]
is an isomorphism in $\D(\FX)$. By the previous lemma we have the isomorphism
\[
\rshom^\bullet_\FX(\CE^\bullet, \R\varGamma'_\FX \CO_\FX) \liso
\rshom^\bullet_\FX(\CE^\bullet, \CO_\FX) 
\]
which reduces us to check that the canonical map
\[
\rshom^\bullet_\FX(\CE^\bullet, \CO_\FX) \otimes^{\LL}_{\CO_\FX} \CG^\bullet \lto
\rshom^\bullet_\FX(\CE^\bullet, \CG^\bullet)
\]
is an isomorphism in $\D(\FX)$. But this is a local problem and we can argue as at the end of the proof of Proposition \ref{perfsd}.
\end{proof}

\begin{thm}\label{fs-asht3} 
Let $\FX$ be a noetherian semi-separated formal scheme. The property (\ref{sdg}) holds in the category $\D_\qct(\FX)$. 
\end{thm}

\begin{proof}
By Proposition \ref{fs-cgenss}, $\D_\qct(\FX)$ is compactly generated. It follows then from \ref{strong} that there exists a set of compact objects $\CS$ such that the smallest triangulated subcategory of $\D_\qct(\FX)$ closed for coproducts containing $\CS$ is the whole category. As compact objects are perfect complexes by Proposition \ref{fs-compsd} and perfect complexes are strongly dualizable by Propostion \ref{fs-perfsd}, the result follows.
\end{proof}

\begin{cor} 
Let $\FX$ be a noetherian semi-separated formal scheme. The category $\D_\qct(\FX)$ is a stable homotopy category in the sense of \cite{hps}. 
\end{cor}

\begin{proof}
By Theorem \ref{fs-asht145}, $\D(\A_\qct(\FX))$ satisfies the properties (\ref{tri}),  (\ref{cop}) and (\ref{brw}), but by \ref{fs-equiv} this category is equivalent to $\D_\qct(\FX)$, so it has the same properties. The rest of the conditions are dealt with in \ref{fs-asht2} and \ref{fs-asht3}.
\end{proof}

\begin{cor}
Let $\FX$ be a noetherian semi-separated formal scheme. The category $\D_\qct(\FX)$ is an \emph{algebraic} stable homotopy category in the sense of \cite{hps}. 
\end{cor}

\begin{proof}
The adjective ``algebraic'' just means that the set of generators is compact which is true by Proposition \ref{fs-compsd}.
\end{proof}

\begin{rem}
Note that the category $\D_\qct(\FX)$ is seldom a \emph{unital} stable homotopy category. Let $\FX = \spf(K[[T]])$ where $K$ is a field. Then, in this case, $\CO'_\FX$ may be represented by the complex $\CO_\FX \to \CM_\FX$ in degrees 0 and 1 with $\CM_\FX := \kappa^*(K((T))^\sim)$. And this complex is not perfect because $\CH^1(\CO'_\FX) = \kappa^*((K((T))/K[[T]])^\sim)$ is not coherent.
\end{rem}

\begin{rem}
We conjecture that the methods similar to those of \cite[\S 2]{bb} and \cite[\S 4]{LN} can be extended to prove that for a semi-separated noetherian formal scheme $\FX$, the category $\D_\qct(\FX)$ is monogenic.
\end{rem}


\end{document}